\documentclass[11pt]{amsart}

\usepackage{epigamath}

\usepackage[english]{babel}

\numberwithin{equation}{section}

\usepackage[all,cmtip]{xy}
\usepackage{tikz-cd}
\usepackage[all]{xy} 
\usepackage{mathtools}
\usepackage{amsmath, amssymb, amsfonts, latexsym, mdwlist, amsthm}
\usepackage{subfig}
\usepackage{graphicx}
\usepackage{wrapfig}

\usepackage{tikz}
\usetikzlibrary{calc,trees,positioning,arrows,chains,shapes.geometric,%
    decorations.pathreplacing,decorations.pathmorphing,shapes,%
    matrix,shapes.symbols}

\tikzset{
>=stealth',
  punktchain/.style={
    rectangle,
    rounded corners,
    draw=black, thick,
    minimum height=3em,
    text centered,
    on chain},
  line/.style={draw, thick, <-},
  element/.style={
    tape,
    top color=white,
    bottom color=blue!50!black!60!,
    minimum width=8em,
    draw=blue!40!black!90, very thick,
    text width=10em,
    minimum height=3.5em,
    text centered,
    on chain},
  every join/.style={->, thick,shorten >=1pt},
  decoration={brace},
  tuborg/.style={decorate},
  tubnode/.style={midway, right=2pt},
}

\usepackage{paralist}
\usepackage[shortlabels]{enumitem} 
\setlist[enumerate,1]{label={\rm(\alph*)}, ref={\rm\alph*}} 

\usepackage{graphicx}

\usetikzlibrary{patterns}

\usepackage{mathtools}

\makeatletter
\newtheorem*{rep@theorem}{\rep@title}
\newcommand{\newreptheorem}[2]{%
\newenvironment{rep#1}[1]{%
 \def\rep@title{#2 \ref{##1}}%
 \begin{rep@theorem}}%
 {\end{rep@theorem}}}
\makeatother

\newtheorem{Thm}{Theorem}[section]
\newreptheorem{Thm}{Theorem}
\newtheorem{Thm*}{Theorem}
\newtheorem{Prop}[Thm]{Proposition}

\newtheorem{Lem}[Thm]{Lemma}

\newreptheorem{Cor}{Corollary}
\newtheorem{Con}[Thm]{Conjecture}
\newtheorem{Conj}{Conjecture}

\newtheorem{thm-int}{Theorem}

\theoremstyle{definition}
\newtheorem{Def-s}[Thm]{Definition}
\newtheorem{Def}[Thm]{Definition}

\theoremstyle{remark}
\newtheorem{Rem}[Thm]{Remark}

\renewcommand\;{\hspace{.7pt}}

\renewcommand\_{^{}_}
\newcommand\To{\longrightarrow}
\newcommand\into{\hookrightarrow}
\newcommand{\Into}{\ensuremath{\lhook\joinrel\relbar\joinrel\rightarrow}}
\newcommand\Onto{\longrightarrow\hspace{-5.5mm}\longrightarrow}

\newcommand\PP{\mathbb P}
\newcommand\C{\mathbb C}
\newcommand\Q{\mathbb Q}
\newcommand\R{\mathbb R}

\newcommand\Z{\mathbb Z}
\newcommand\m{\mathfrak m}
\newcommand\cA{\mathcal A}
\newcommand\cD{\mathcal D}
\newcommand\cH{\mathcal H}
\newcommand\cO{\mathcal O}
\newcommand\cI{\mathcal I}
\newcommand\cL{\mathcal L}
\newcommand\cF{\mathcal F}

\newcommand\cG{\mathcal G}
\newcommand\cT{\mathcal T}
\newcommand\ch{\operatorname{ch}}
\newcommand\Aut{\operatorname{Aut}}
\newcommand\Sym{\operatorname{Sym}}

\newcommand\Pic{\operatorname{Pic}}
\newcommand\Ext{\operatorname{Ext}}
\newcommand\rk{\operatorname{rank}}

\newcommand\cok{\operatorname{cok}}
\newcommand\js{\operatorname{JS}}
\newcommand\supp{\operatorname{supp}}
\newcommand\vi{v}

\newcommand\ext{\mathcal E\hspace{-1pt}xt}

\newcommand\dbar{\overline\partial}

\renewcommand\={\ =\ }

\def\abs#1{\left\lvert#1\right\rvert}

\newcommand{\ignore}[1]{}

\def\isolow{\vbox to 0pt{\vss\hbox{$\scriptstyle\sim$}\vskip-1.5pt}}

\newcommand\tors{\operatorname{tors}}
\newcommand\vir{\operatorname{vir}}
\newcommand\slope{\operatorname{slope}}
\newcommand\prim{\operatorname{prim}}
\newcommand\im{\operatorname{im}}
\newcommand\codim{\operatorname{codim}}
\newcommand\Hilb{\operatorname{Hilb}}
\newcommand\disc{\operatorname{disc}}
\newcommand\Spec{\operatorname{Spec}}

\EpigaVolumeYear{7}{2023} \EpigaArticleNr{15} \ReceivedOn{July 21, 2022}
\InFinalFormOn{January 4, 2023}
\AcceptedOn{February 7, 2023}

\title{Curve counting and S-duality}
\titlemark{Curve counting and S-duality}

\author{S.~Feyzbakhsh}
\address{Department of Mathematics, Imperial College, London SW7 2AZ, United Kingdom}
\email{s.feyzbakhsh@imperial.ac.uk}

\author{R.\,P.~Thomas}
\address{Department of Mathematics, Imperial College, London SW7 2AZ, United Kingdom}
\email{richard.thomas@imperial.ac.uk}

\authormark{S.~Feyzbakhsh and R.\,P.~Thomas}

\AbstractInEnglish{We work on a projective threefold $X$ which satisfies the Bogomolov--Gieseker conjecture of Bayer--Macr\`i--Toda, such as $\mathbb P^3$ or the quintic threefold.

We prove certain moduli spaces of 2-dimensional torsion sheaves on $X$ are smooth bundles over Hilbert schemes of ideal sheaves of curves and points in $X$.

When $X$ is Calabi--Yau, this gives a simple wall-crossing formula expressing curve counts (and so ultimately Gromov--Witten invariants) in terms of counts of D4-D2-D0 branes. These latter invariants are predicted to have modular properties which we discuss from the point of view of S-duality and Noether--Lefschetz theory.} 



\acknowledgement{We acknowledge the support of an EPSRC postdoctoral fellowship EP/T018658/1, an EPSRC grant EP/R013349/1 and a Royal Society research professorship.}

\begin{document}


\maketitle

\begin{prelims}

\DisplayAbstractInEnglish

\bigskip

\DisplayKeyWords

\medskip

\DisplayMSCclass

\end{prelims}


\newpage

\setcounter{tocdepth}{1}

\tableofcontents


\section{Introduction}

A naive way to relate different moduli spaces of coherent sheaves on an algebraic variety is to replace sheaves by cokernels of their sections (after twisting by an appropriately positive line bundle to ensure the sheaves have predictable numbers of sections).
This would give useful relations between enumerative invariants counting sheaves, but it very rarely works due to stability issues.

This is one of a series of papers \cite{F22, FT, rk2, FT3} showing it can be made to work --- modulo the wall-crossing formulae required to move to a stability condition for which such cokernels are stable --- on a threefold $X$ satisfying (a weakening of) Bayer--Macr\`i--Toda's Bogomolov--Gieseker conjecture \cite{BMT}. When $X$ is Calabi--Yau, this gives relations between invariants involving unwieldy formulae. In this paper we concentrate on the one case where the formulae are very simple. They equate the (virtual) counting of ideal sheaves (of curves and points) with the counting of 2-dimensional pure sheaves,
\begin{equation}\label{hash}
\#(\text{ideal sheaves})\=c\cdot\#(\text{2-dimensional pure sheaves}),
\end{equation}
with $c$ an explicit topological constant. The formula reflects the fact that in this case we are ultimately able to prove (by crossing many walls) that no wall-crossing is necessary --- the cokernels are stable and are the only stable sheaves of the same topological type. Thus the relationship between the moduli spaces is also very simple, with one being a projective bundle over the other. We explain below why this is so surprising (to us).

The counts of ideal sheaves on the left-hand side of \eqref{hash} are heavily studied, being equivalent, via the MNOP conjecture \cite{MNOP}, to the Gromov--Witten theory of $X$. On the right-hand side, we have counts of ``D4-D2-D0 branes'', which the string-theoretic ``S-duality conjecture'' predicts are coefficients of (vector-valued mock) modular forms. Thus \eqref{hash} opens up the potential of governing the Gromov--Witten theory of $X$ by modular forms. This is discussed in Section~\ref{modular}, alongside a conjectural approach to proving S-duality geometrically using Noether--Lefschetz theory.

In the sequel \cite{rk2, FT3} we apply the same methods to DT invariants counting higher-rank sheaves. Here the wall-crossing formulae are much more complicated, but ultimately express arbitrary rank DT invariants in terms of the same data. That is, we can express them in terms of the rank 1 invariants on the left-hand side of \eqref{hash} (or, by the MNOP conjecture, the Gromov--Witten invariants of $X$). Or we can express them in terms of the (conjecturally modular) rank 0 invariants counting D4-D2-D0 branes on the right-hand side of \eqref{hash}.

\subsection*{Statement of results} Let $(X, \cO(1))$ be a smooth polarised complex projective threefold. Gieseker and slope (semi)stability of sheaves will always be defined by $H:=c_1(\cO(1))$.

Fix $\beta$ in the image of $H^4(X,\Z)\to H^4(X,\Q)$ and $m\in\Z$. Let $I_m(X,\beta)$ denote the Hilbert scheme of subschemes $C\subset X$ of dimension at most 1
and topological type
$$
\ch_2(\cO_C)\ =\ [C]\ =\ \beta \quad\text{and}\quad \ch_3(\cO_C)\ =\ m.
$$
It is the moduli space of ideal sheaves $I_C$. Torsion-free sheaves of the same Chern character
$$
v\ :=\ (1, 0,-\beta, -m)\ \in\ \textstyle{\bigoplus_{i=0\,}^3}H^{2i}(X,\Q)
$$
are all of the form $I_C\otimes T$, where $T$ is a line bundle with torsion $c_1(T)\in H^2(X,\Z)$. Their moduli space is $I_m(X,\beta)\times\Pic\_0(X)$, where
$$
\Pic\_0(X)\ :=\ \left\{T\in\Pic(X)\ \colon\ c_1(T)\,=\,0\,\in\,H^2(X,\Q)\right\}.
$$
For $n\gg0$ the generic section
\begin{equation}\label{secshun}
s\,\colon\,\mathcal{O}_X(-n) \To I_C\otimes T
\end{equation}
has cokernel a rank 0 Gieseker semistable 2-dimensional sheaf $\cok(s)$ of Chern character
\begin{equation}\label{vn}
\vi_n\ :=\ v-\ch(\cO(-n))\=\left(0,\, nH,\, -\beta -\frac12n^2H^2 ,\, -m + \frac16n^3H^3\right).
\end{equation}
Let $M_{X,H}(\vi_n)$ denote the moduli space of Gieseker semistable sheaves of class $v_n$.

\begin{Thm*}\label{theorem.1}
For $X$ satisfying the Bogomolov--Gieseker conjecture of \,\cite{BMT} and $n\gg0$,
\begin{itemize}
\item $\cok(s)\otimes L$ is slope and Gieseker stable for any $s\ne0$ and any $L\in\Pic\_0(X)$;
\item all points of $M_{X,H}(\vi_n)$ are of the form $\cok(s)\otimes L$ for unique data $((I_C\otimes T,s),L)$;
\item mapping $\cok(s)\otimes L$ to $(I_C,T,L)$ defines a morphism
$$
M_{X,H}(\vi_n)\To I_m(X,\beta)\times\Pic\_0(X)\times\Pic\_0(X)
$$
which is a smooth projective bundle with fibre $\PP^{\;\chi(v(n))-1}$.
\end{itemize}
\end{Thm*}

We use the \emph{weak stability conditions} of Bayer--Macr\`{i}--Toda \cite{BMT} on the bounded derived category of coherent sheaves $\cD(X)$. Their conjectural Bogomolov--Gieseker inequality for weak semistable objects is the key to proving the existence of Bridgeland stability conditions on $\cD(X)$, and is now known to hold for many threefolds, see \cite{BMS, KosekiAb, Kosekinef, Ko20, Li, Li.Fano, MP, Ma, ScQ}, such as $\PP^3$ or the quintic threefold. In fact we only need its weakening stated in Conjecture~\ref{BG} below.

For weak stability conditions in the large volume limit, semistable objects $F$ of Chern character $\vi_n$ are slope semistable sheaves, giving the moduli space $M_{X,H}(\vi_n)$. Several applications of the Bogomolov--Gieseker inequality show there is a single wall of instability for the class $v_n$. Below this there are no semistable objects, while on the wall $F$ is destabilised by an exact triangle
$$
I_C\otimes T\otimes L\To F\To L(-n)[1]
$$
expressing it as the cokernel of a section \eqref{secshun} tensored by $L\in\Pic\_0(X)$.

%
%
%
%
%
%
%
%
%

By \cite{GST} such sheaves $F=\cok(s)\otimes L$ are all of the form
\begin{equation}\label{form}
\iota_*(\cI_C\otimes L')
\end{equation}
for some subscheme $C$ of a divisor $\iota\colon D\into X$ in the linear system $|T(n)|$ and $L'=T\otimes L$ in $\Pic\_0(X)$. (Here $\cI_C$ denotes the ideal sheaf of $C$ on $D$ rather than $X$, and the $H_2$ class of the pushforward of $[C]$ to $X$ is $\beta$.) So, remarkably, such sheaves are always Gieseker and slope stable, and \emph{these are the only sheaves} in $M_{X,H}(\vi_n)$.

In other words, semistable sheaves with Chern character $\vi_n$ \emph{must have rank $1$ on their support}. This came as a surprise to us. The support $D$ can have reducible components which are non-reduced. For instance, consider one of the form $D=D_1\cup rD_2$ for $r>1$. Then there are sheaves with Chern character $\vi_n$ which have support on the reduced divisor $D_1\cup D_2$ but rank $r$ on $D_2$. Theorem~\ref{theorem.1} says they are necessarily \emph{unstable}.\footnote{Of course, stable sheaves do exist with rank $r>1$ on $D_2$, but if they have the same $\ch_1,\,\ch_2$ as $v_n$, then by Theorem~\ref{theorem.1} they have $\ch_3<-m+\frac16n^3H^3$.} (Furthermore, there is a \emph{unique} $L'\in\Pic\_0(X)$ such that after $\otimes(L')^*$, this rank 1 sheaf on $D$ is \emph{anti-effective}: its dual has a section. In particular, if $\beta\cdot H<0$, then $M_{X,H}(\vi_n)$ must be empty.) We do not know of other situations in which whole series of different moduli spaces of sheaves have been systematically shown to have such a simple geometric relation as in Theorem~\ref{theorem.1}.

\medskip
Suppose now that $X$ is a Calabi--Yau threefold: $K_X\cong\cO_X$ and $H^1(\cO_X)=0$. Since semistability equals stability for our moduli spaces, they carry symmetric obstruction theories and virtual cycles of dimension 0, see \cite{HT, MNOP, Th}, and degrees\vspace{1mm}
\begin{itemize}
\setlength\itemsep{5pt}
\item $I_{m,\beta}(X)\,:=\,\int_{[I_m(X,\beta)]^{\vir}}1\in\Z$,
\item $\Omega_{\vi_n}(X)\,:=\,\int_{[M_{X,H}(\vi_n)]^{\vir}}1\in\Z$.
\end{itemize}
The first is the count of (ideal sheaves of) curves and points on $X$ conjectured by Maulik--Nekrasov--Okounkov--Pandharipande to be equivalent to the Gromov--Witten theory of $X$; see \cite{MNOP}. This MNOP conjecture has now been proved for most Calabi--Yau threefolds by Pandharipande and Pixton \cite{PP}.
The second counts D4-D2-D0-branes, or 2-dimensional torsion sheaves, and is subject to the famous S-duality conjectures of physicists.

Write these invariants as Behrend-weighted Euler characteristics; see \cite{Be}. As Theorem~\ref{theorem.1} gives a smooth fibration $M_{X,H}(v_n)\to I_m(X,\beta)$ with fibres of signed Euler characteristic
$$
e_n\ :=\ (-1)^{\chi(\vi(n))-1}\chi(\vi(n))\left(\#H^2(X,\Z)_{\tors}\right)^2,
$$
an immediate corollary is the following precise form of equation \eqref{hash}.

\begin{Thm*} \label{Theorem 2}
Fix $\beta,\,m$, then $n\gg0$. Suppose Conjecture~\ref{BG} holds on $X$. Then
\begin{equation} \label{Theorem 2 equality}
    \Omega_{\vi_n}(X)\ =\ e_n\cdot I_{m,\beta}(X).
\end{equation}
\end{Thm*}

In Section~\ref{modular} we discuss the conjectural modular properties of
the invariants $\Omega_{v_n}$ from two points of view: (i) S-duality from physics and (ii) Noether--Lefschetz theory \cite{MPa}.

\subsection*{Acknowledgements} We thank Arend Bayer, Luis Garcia, Chunyi Li, Jan Manschot, Davesh Maulik, Rahul Pandharipande and an anonymous referee for their generous help with this paper. Our intellectual debt to Yukinobu Toda is described in Section~\ref{related}.

%
%
%

\section{Weak stability conditions}
Let $(X,\mathcal O(1))$ be a smooth polarised complex threefold and $H = c_1(\mathcal O(1))$. Denote the bounded derived category of coherent sheaves on $X$ by $\cD(X)$ and its Grothendieck group by $K(X):=K(\cD(X))$. In this section we review the notion of a weak stability condition on $\cD(X)$. The main references are \cite{BMT,BMS}.

We define the $\mu\_H$-slope of a coherent sheaf $E$ on $X$ to be
$$
\mu\_H(E)\ :=\ \left\{\!\!\begin{array}{cc} \frac{\ch_1(E)\cdot H^2}{\ch_0(E)H^3} & \text{if }\ch_0(E)\ne0, \\
+\infty & \text{if }\ch_0(E)=0. \end{array}\right.
$$
Associated to this slope, every sheaf $E$ has a Harder--Narasimhan filtration. Its graded pieces have slopes whose maximum we denote by $\mu_H^+(E)$ and minimum by $\mu_H^-(E)$.

For any $b \in \mathbb{R}$, let $\cA(b)\subset\cD(X)$ denote the abelian category of complexes
	\begin{equation}\label{Abdef}
	\mathcal{A}(b)\ =\ \left\{E^{-1} \xrightarrow{\,d\,} E^0 \ \colon\ \mu_H^{+}(\ker d) \leq b \,,\  \mu_H^{-}(\cok d) > b \right\}. 
	\end{equation}
Then $\cA(b)$ is the heart of a t-structure on $\cD(X)$ by \cite[Lemma 6.1]{Br}. Let $w\in\R\setminus\{0\}$. On $\cA(b)$ we have the slope function\footnote{This is called $\nu_{b,w}$ in \cite[Equation 7]{BMT}, but we reserve $\nu_{b,w}$ for its rescaling \eqref{scale}.}
\begin{equation*}
N_{b,w}(E)\ :=\ \left\{\!\!\begin{array}{cc} \frac{w\ch_2^{bH}(E)\cdot H - \frac{1}{6}w^3\ch_0(E)H^3}{w^2\ch_1^{bH}(E)\cdot H^2} & \text{if }\ch_1^{bH}(E)\cdot H^2\ne0, \\
+\infty & \text{if }\ch_1^{bH}(E)\cdot H^2=0, \end{array}\right.
\end{equation*}
where $\ch^{bH}(E):=\ch(E)e^{-bH}$. This defines a Harder--Narasimhan filtration on $\cA(b)$ by \cite[Lemma~3.2.4]{BMT}. It will be convenient to replace this by
\begin{equation}\label{scale}
\nu\_{b,w}\ :=\ \sigma N_{b,\sigma}+b, \quad\text{where }\sigma:=\sqrt{6\left(w-\frac{b^2}{2}\right)},
\end{equation}
for $w>\frac{1}{2}b^2$. This is because
\begin{equation}\label{noo}
\nu\_{b,w}(E)\ =\ \left\{\!\!\begin{array}{cc} \frac{\ch_2(E)\cdot H - w\ch_0(E)H^3}{\ch_1^{bH}(E)\cdot H^2}
 & \text{if }\ch_1^{bH}(E)\cdot H^2\ne0, \\
+\infty & \text{if }\ch_1^{bH}(E)\cdot H^2=0 \end{array}\right.
\end{equation}
has a denominator that is linear in $b$ and numerator linear in $w$, so the walls of $\nu_{b,w}$-instability will turn out to be \emph{linear}; see Proposition~\ref{prop. locally finite set of walls}. Note that if $\ch_i(E)\cdot H^{n-i} = 0$ for $i=0,1,2$, the slope $\nu_{b,w}(E)$ is defined by \eqref{noo} to be $+\infty$. Since \eqref{scale} only rescales and adds a constant, it defines the same Harder--Narasimhan filtration as $N_{b,\sigma}$, so it too defines a weak stability condition on $\cA(b)$. 

\begin{Def}
Fix $w>\frac{1}{2}b^2$. We say $E\in\cD(X)$ is $\nu\_{b,w}$-(semi)stable if and only if
\begin{itemize}
\item $E[k]\in\cA(b)$ for some $k\in\Z$, and
\item $\nu\_{b,w}(F)\,(\le)\,\nu\_{b,w}(E[k]/F)$ for all non-trivial subobjects $F\hookrightarrow E[k]$ in $\cA(b)$.
\end{itemize}
Here $(\le)$ denotes $<$ for stability and $\le$ for semistability.
\end{Def}

\begin{Rem}\label{heart}
Given $(b,w) \in \mathbb{R}^2$ with $w> \frac{1}{2}b^2$, the argument in \cite[Proposition 5.3]{Br.stbaility} describes $\cA(b)$. It is generated by the $\nu_{b,w}$-stable two-term complexes $E = \{E^{-1} \to E^0\}$ in $\cD(X)$ satisfying the following conditions on the denominator and numerator of $\nu_{b,w}$ from \eqref{noo}:
\begin{enumerate}
    \item $\ch_1^{bH}(E)\cdot H^2 \geq 0$, and
    \item $\ch_2(E)\cdot H - w\ch_0(E)H^3 \geq 0$ if $\ch_1^{bH}(E)\cdot H^2 = 0$. 
\end{enumerate}
That is, $\cA(b)$ is the extension-closure of the set of these complexes. 
\end{Rem}

We recall the conjectural strong Bogomolov--Gieseker inequality of \cite[Conjecture 1.3.1]{BMT}, rephrased in terms of the rescaling \eqref{scale}.

\begin{Con}  \label{conjecture}
	For $\nu\_{b,w}$-semistable $E\in\cA(b)$ with $\ch_2^{bH}(E)\cdot H =\big(w -\frac{1}{2}b^2\big)\ch_0(E)H^3$, 
\begin{equation}\label{BGineq}
		\ch_3^{bH}(E)\ \leq\ \left(\frac{w}{3} - \frac{b^2}{6}\right) \ch_1^{bH}(E)\cdot H^2.
\end{equation}
	\end{Con}
	
Although this conjecture is known \emph{not} to hold for all classes on all threefolds \cite{Sc}, it is possible it always holds for the special classes required to prove Theorem~\ref{theorem.1}. Setting
$$
b_0\ :=\ -\frac{n}{2} - \frac{\beta\cdot H}{nH^3}\,, \quad
w_f\ := \dfrac{n^2}{4} -\frac{\beta\cdot H}{H^3} - \dfrac{3m}{nH^3}  - \left(\frac{\beta\cdot H}{nH^3}\right)^{\!2},
$$
we require the following.

\begin{Conj}\label{BG}{\samepage
Conjecture~\ref{conjecture} holds in case \eqref{BG-1} below, and for \eqref{BG-2} when $\beta\cdot H > 0$.
\begin{enumerate}[label={\rm(\roman*)}, ref={\rm\roman*}]
\item\label{BG-1} $b =b_0$, $w\in(w_f-\epsilon,w_f]$ for some $0< \epsilon \ll 1$, and $\ch(E)=\vi_n$.
\item\label{BG-2} $b = \ch_2(E)\cdot H - \frac{1}{2H^3},\ w = b^2 + \frac{\ch_2(E)\cdot H}{H^3}$, and $E$ a torsion-free sheaf with
$$
\ch_0(E)\,=\,1, \quad \ch_1(E)\cdot H^2\,=\,0, \quad 
-\ch_2(E)\cdot H\,\in\,\left[\beta\cdot H,\,2\beta\cdot H\right].
$$ 
\end{enumerate}}
\end{Conj}

In fact an even weaker version of~\ref{BG} in enough to conclude Theorems~\ref{theorem.1} and~\ref{Theorem 2}, as shown in \cite[Section 3.1]{rk2}.
Conjecture~\ref{conjecture} follows from \cite[Conjecture 4.1]{BMS}, which has now been proved in the following cases: 
\begin{itemize}
\item $X$ is projective space $\mathbb{P}^3$ (see \cite{Ma}), the quadric threefold (see \cite{ScQ}) or, more generally, any Fano threefold of Picard rank 1 (see \cite{Li.Fano}),
\item $X$ is an abelian threefold (see \cite{MP}), a Calabi--Yau threefold of abelian type (see \cite{BMS}), a Kummer threefold (see \cite{BMS}), or a product of an abelian variety and $\PP^n$ (see \cite{KosekiAb}),
\item $X$ has nef tangent bundle (see \cite{Kosekinef}),
\item $X$ is one of the Calabi--Yau threefolds considered in \cite{Ko20}; with some work, one can show the weakening of Conjecture~\ref{conjecture} proved in \cite[Theorem 1.2]{Ko20} is still strong enough to give Theorem~\ref{theorem.1} for $n\gg0$, and
\item $X$ is a quintic threefold (see \cite{Li}), or a (2,4) complete intersection in $\PP^5$ (see \cite{Liu}), and $(b,w)$ are described below.
\end{itemize}

\begin{Thm}[\textit{cf.}~\protect{\cite[Theorem 2.8]{Li}, \cite[Theorem 2.14]{Liu}}] \label{Li}
Conjecture~\ref{conjecture} holds on a smooth quintic threefold or a $(2,4)$ complete intersection in $\PP^5$ when $(b,w)$ satisfy 
\begin{equation}\label{in for b, w}
w\ >\ \frac{1}{2} b^2 + \frac{1}{2}\left(b - \lfloor b \rfloor\right)\left(\lfloor b \rfloor+1 - b \right). 
\end{equation}
In particular, Conjecture~\ref{BG} holds for $n \gg 0$. 
\end{Thm}

\begin{proof}
Using the notation $(\alpha,\beta)$ for our $(w,b)$, \cite[Theorem 2.8]{Li} and \cite[Theorem 2.14]{Liu} prove that \eqref{in for b, w} implies \cite[Conjecture 4.1]{BMS}. This gives Conjecture~\ref{conjecture}, so we only need to check that the parameters in Conjecture~\ref{BG} satisfy \eqref{in for b, w}. 

For the parameters in the first part of Conjecture~\ref{BG}, we have 
$$
    w_f - \frac{b_0^2}{2}\ =\ \frac{n^2}{8} - \frac{3\beta\cdot H}{2H^3} - \frac{3m}{nH^3} - \frac{3}{2} \left(\frac{\beta\cdot H}{nH^3}\right)^{\!2},
$$
which for $n\gg0$ satisfies 
$$
 w_f - \frac{b_0^2}{2}>\ \frac12\ \geq\ \frac{1}{2}\left(b_0 - \lfloor b_0 \rfloor\right)\left(\lfloor b_0 \rfloor +1- b_0 \right).
$$
Thus there exists an $0 < \epsilon \ll 1$ such that $(b_0,w)$ satisfies \eqref{in for b, w} whenever $w\in(w_f -\epsilon,w_f]$.  

For the second part of Conjecture~\ref{BG}, use the obvious inequality
$$
2x\left(x-\frac{1}{H^3}\right)\ >\ \frac{1}{H^3}\left(1-\frac{1}{H^3}\right) \quad\text{for }x\ \ge\ 1.
$$
Rearranging gives
$$
\frac12\left(-x-\frac1{2H^3}\right)^{\!2}-\frac x{H^3}\ >\ \frac1{4H^3}\left(1-\frac1{2H^3}\right).
$$
Substituting in $x=-\ch_2(E)\cdot H \ge \beta\cdot H \ge 1$ and $b = \ch_2(E)\cdot H - \frac{1}{2H^3}$ makes this
$$
\frac{b^2}2+\frac{\ch_2(E)\cdot H}{H^3}\ >\ \frac1{4H^3}\left(1-\frac1{2H^3}\right).
$$
For $w = b^2 + \frac{\ch_2(E)\cdot H}{H^3}$ this is 
$$
w-\frac{b^2}2\ >\ \frac12\left(1-\frac1{2H^3}\right)\frac1{2H^3}\ =\ \frac12\left(b - \lfloor b \rfloor\right)\left(\lfloor b\rfloor-b+1\right)
$$
since $\ch_2(E)\cdot H\in\Z$ for $E$ of rank 1 with $\ch_1(E)=0$. Thus \eqref{in for b, w} holds for this $(b,w)$.
\end{proof}

In Figure~\ref{projetcion} we plot the $(b,w)$-plane simultaneously with the image of the projection map
\begin{eqnarray*}
	\Pi\colon\ K(X) \setminus \left\{E \colon \ch_0(E) = 0\right\}\! &\longrightarrow& \R^2, \\
	E &\longmapsto& \!\!\left(\frac{\ch_1(E)\cdot H^2}{\ch_0(E)H^3}\,,\, \frac{\ch_2(E)\cdot H}{\ch_0(E)H^3}\right).
\end{eqnarray*}
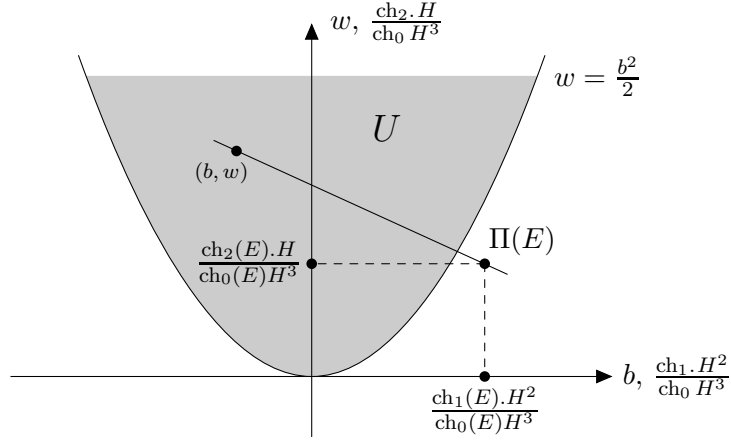
\begin{figure}[h]
	\begin{centering}
		\definecolor{zzttqq}{rgb}{0.27,0.27,0.27}
		\definecolor{qqqqff}{rgb}{0.33,0.33,0.33}
		\definecolor{uququq}{rgb}{0.25,0.25,0.25}
		\definecolor{xdxdff}{rgb}{0.66,0.66,0.66}
		
		\begin{tikzpicture}[line cap=round,line join=round,>=triangle 45,x=1.0cm,y=1.0cm]
		
		\draw[->,color=black] (-4,0) -- (4,0);
		\draw  (4, 0) node [right ] {$b,\,\frac{\ch_1\!.\;H^2}{\ch_0H^3}$};


		\fill [fill=gray!40!white] (0,0) parabola (3,4) parabola [bend at end] (-3,4) parabola [bend at end] (0,0);
		
		\draw  (0,0) parabola (3.1,4.27); 
		\draw  (0,0) parabola (-3.1,4.27); 
		\draw  (3.8 , 3.6) node [above] {$w= \frac{b^2}{2}$};
		
		

		\draw[->,color=black] (0,-.8) -- (0,4.7);
		\draw  (1, 4.3) node [above ] {$w,\,\frac{\ch_2\!.\;H}{\ch_0H^3}$};

		
		\draw [dashed, color=black] (2.3,1.5) -- (2.3,0);
		\draw [dashed, color=black] (2.3, 1.5) -- (0, 1.5);
		\draw [color=black] (2.6, 1.36) -- (-1.3, 3.14);
		
		\draw  (2.8, 1.8) node {$\Pi(E)$};
		\draw  (1, 3) node [above] {\Large{$U$}};
		\draw  (0, 1.5) node [left] {$\frac{\ch_2(E)\cdot H}{\ch_0(E)H^3}$};
		\draw  (2.3 , 0) node [below] {$\frac{\ch_1(E)\cdot H^2}{\ch_0(E)H^3}$};
		\begin{scriptsize}
		\fill (0, 1.5) circle (2pt);
		\fill (2.3,0) circle (2pt);
		\fill (2.3,1.5) circle (2pt);
		\fill (-1,3) circle (2pt);
		\draw  (-1.2, 2.96) node [below] {$(b,w)$};
		
		\end{scriptsize}
		
		\end{tikzpicture}
		
		\caption{$(b,w)$-plane and the projection $\Pi(E)$ when $\ch_0(E)>0$}
		
		\label{projetcion}
		
	\end{centering}
\end{figure}

\noindent Note that for any weak stability condition $\nu_{b,w}$, the pair $(b,w)$ is in the shaded open subset
\begin{equation}\label{Udef}
U \,:= \,\left\{(b,w) \in \mathbb{R}^2 \colon w > \frac{b^2}{2}  \right\}.
\end{equation}
Conversely, the image $\Pi(E)$ of $\nu_{b,w}$-semistable objects $E$ with $\ch_0(E)\ne0$ is \emph{outside} $U$ due to the classical Bogomolov--Gieseker-type inequality
for the $H$-discriminant,
\begin{equation}\label{discr}
	\Delta_H(E)\ =\  \left(\!\ch_1(E)\cdot H^2\right)^2 -2 (\ch_0(E)H^3)(\ch_2(E)\cdot H)\ \ge\ 0,
\end{equation}
proved for $\nu_{b,w}$-semistable objects $E$ in \cite[Theorem 3.5]{BMS}.\footnote{\cite[Theorem 3.5]{BMS} states \eqref{discr} with $\ch$ replaced by $\ch^{bH}$, but the result is still $\Delta_H(E)$. We use the stronger Bogomolov inequality $\ch_1(E)^2\cdot H-2\ch_0(E)(\ch_2(E)\cdot H)\ge0$ for $\mu\_H$-semistable sheaves in \eqref{condition 2}.} By Remark~\ref{heart} they lie to the right of (or on) the vertical line through $(b,w)$ if $\ch_0(E)>0$, to the left if $\ch_0(E)<0$, and at infinity if $\ch_0(E)=0$. The slope $\nu_{b,w}(E)$ of $E$ is the gradient of the line connecting $(b,w)$ to $\Pi(E)$.

Objects of $\cD(X)$ give the space of weak stability conditions a wall and chamber structure, explained in \cite[Proposition 4.1]{FT} for instance.

\begin{Prop}[{Wall and chamber structure}]\label{prop. locally finite set of walls}
	Fix an object $E \in \mathcal{D}(X)$ such that the vector $(\ch_0(E), \ch_1(E)\cdot H^2,$ $\ch_2(E)\cdot H)$ is non-zero. There exists a set of lines $\{\ell_i\}_{i \in I}$ in $\mathbb{R}^2$ such that the segments $\ell_i\cap U$ $($called ``\emph{walls}''\,$)$ are locally finite and satisfy:  
	\begin{enumerate}
	    \item If\, $\ch_0(E)\ne0$, then all lines $\ell_i$ pass through $\Pi(E)$.
	    \item If\, $\ch_0(E)=0$, then all lines $\ell_i$ are parallel of slope $\frac{\ch_2(E)\cdot H}{\ch_1(E)\cdot H^2}$.
	   		\item The $\nu\_{b,w}$-$($semi\,$)$stability of $E$ is unchanged as $(b,w)$ varies within any connected component $($called a ``\emph{chamber}''\,$)$ of $U \setminus \bigcup_{i \in I}\ell_i$.
		\item For any wall $\ell_i\cap U$ there are a $k_i \in \mathbb{Z}$ and a map $f\colon E[k_i] \to F$ in $\cD(X)$ such that
\begin{itemize}
\item for any $(b,w) \in \ell_i \cap U$ the objects $E[k_i],\,F$ lie in the heart $\cA(b)$; 
\item $E[k_i]$ is $\nu\_{b,w}$-semistable with $\nu\_{b,w}(E)=\nu\_{b,w}(F)=\,\slope(\ell_i)$ constant on the wall $\ell_i \cap U$, and
\item $f$ is a surjection $E[k_i] \twoheadrightarrow F$ in $\cA(b)$ which strictly destabilises $E[k_i]$ for $(b,w)$ in one of the two chambers adjacent to the wall $\ell_i$.
\end{itemize} 
	\end{enumerate} 
\end{Prop}
	
	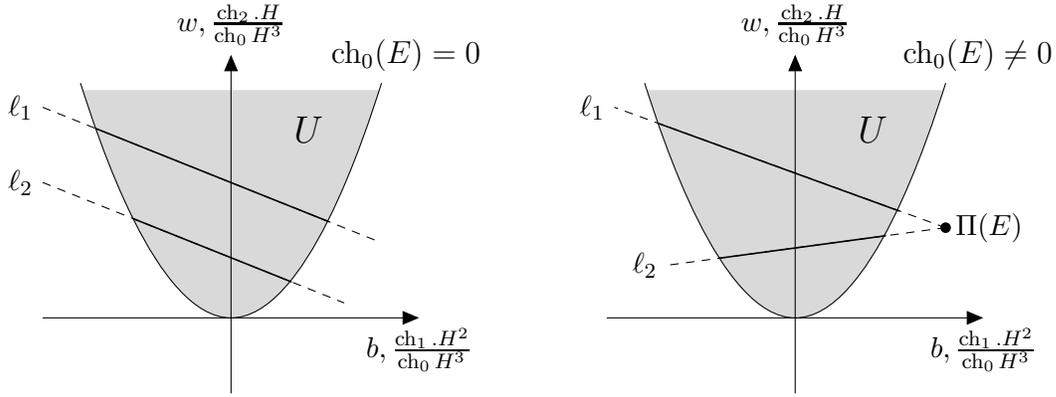
\begin{figure} [h]
	\begin{centering}
		
		\begin{tikzpicture}[line cap=round,line join=round,>=triangle 45,x=1.0cm,y=1.0cm]
	
		\draw[->,color=black] (-10.5,0) -- (-5.5,0);
		\draw[->,color=black] (-3,0) -- (2,0);
		
		\fill [fill=gray!30!white] (-0.5,0) parabola (1.47, 3.03) parabola [bend at end] (-2.47,3.03) parabola [bend at end] (-0.5,0);
		
		\fill [fill=gray!30!white] (-8,0) parabola (-6.03, 3.03) parabola [bend at end] (-9.97,3.03) parabola [bend at end] (-8,0);

		\draw[->,color=black] (-8,-1) -- (-8,3.5);
		\draw[->,color=black] (-0.5,-1) -- (-0.5,3.5);

		\draw [] (-0.5,0) parabola (1.5,3.12); 
		\draw [] (-0.5,0) parabola (-2.5,3.12); 
		\draw [] (-8,0) parabola (-10,3.12); 
		\draw [] (-8,0) parabola (-6,3.12);

		\draw[color=black, dashed] (-10.5,2.8) -- (-6,1);
		\draw[color=black, dashed] (-10.5,1.8) -- (-6.5,0.2);

       \draw[color=black,semithick] (-9.8,2.52) -- (-6.7,1.28);
       \draw[color=black,semithick] (-9.3,1.32) -- (-7.2,.48);
       
		\draw (-10.5,1.8) node [left] {$\ell_2$};
		\draw (-10.5,2.8) node [left] {$\ell_1$};

		\draw (.8,3.5) node [right] {\large{$\ch_0(E) \neq 0$}};
		\draw (-6.8,3.5) node [right] {\large{$\ch_0(E) = 0$}};
		\draw (-5.5,0) node [below] {$b, \frac{\ch_1\cdot H^2}{\ch_0H^3}$};
		\draw (-8,3.5) node [above] {$w, \frac{\ch_2\cdot H}{\ch_0H^3}$};
		
		\draw (2,0) node [below] {$b, \frac{\ch_1\cdot H^2}{\ch_0H^3}$};
		\draw (-0.5,3.5) node [above] {$w, \frac{\ch_2\cdot H}{\ch_0H^3}$};
	
		\draw (-7.3,2.5) node [right] {\Large{$U$}};
		\draw (0.2,2.5) node [right] {\Large{$U$}};

		\draw (1.5, 1.2) node [right] {$\Pi(E)$};
		
		\draw[color=black, dashed] (1.5, 1.2) -- (-2.9,2.8);
		\draw[color=black, dashed] (1.5, 1.2) -- (-2.2, .7);
		
		\draw (-2.9,2.8) node [left] {$\ell_1$};
		\draw (-2.2, .7) node [left] {$\ell_2$};

		\draw[color=black, semithick] (-2.3 ,2.58) -- (.88,1.423);
		\draw[color=black, semithick] (-1.5,.795) -- (0.7,1.092);

		\begin{scriptsize}
	
%
%
%

		\fill [color=black] (1.5,1.2) circle (2pt);
		
%

		
%
%
%
%
%
%
		

		\end{scriptsize}
		
		\end{tikzpicture}
		
		\caption{The line segments $\ell_i \cap U$ are walls for $E$.}
		
		\label{wall.figure}
		
	\end{centering}
	
\end{figure}

\section{From sheaves to Joyce--Song pairs}
Let $(X, \mathcal{O}(1))$ be a smooth polarised complex threefold, and let $H := c_1(\mathcal{O}(1))$. Fix $\beta$ in the image of $H^4(X, \mathbb{Z}) \to H^4(X, \mathbb{Q})$ and $m \in \mathbb{Z}$. In this section we investigate walls of instability for sheaves of Chern character
\begin{equation*}
	\textstyle{\vi_n := \left(0, nH,\,-\beta -\dfrac{1}{2}n^2H^2 ,\, -m + \dfrac{1}{6}n^3H^3 \right)}
\end{equation*}
when $n \gg 0$. For any sheaf $F$ of rank 0, we define its $\nu\_H$-slope as 
\begin{equation}\label{nuslope}
\nu\_H(F)\ :=\ \left\{\!\!\begin{array}{cc} \frac{\ch_2(F)\cdot H}{\ch_1(F)\cdot H^2} & \text{if }\ch_1(F)\cdot H^2\ne0, \\
+\infty & \text{if }\ch_1(F)\cdot H^2=0. \end{array}\right.
\end{equation}
We say that a sheaf $F$ of rank 0 is slope (semi)stable if for all non-trivial quotients $F\to\hspace{-3mm}\to F'$, one has $\nu\_H(F)\,(\leq)\, \nu\_H(F')$.

This section is devoted to proving the following half of Theorem~\ref{theorem.1}.

\begin{Thm}\label{Theorem. part 1}
Fix $\beta\in\im\big(H^4(X,\Z)\to H^4(X,\Q)\big),\ m\in\Z,\ n\gg0$, and suppose Conjecture~\ref{BG} holds on $X$. 
Then any slope semistable sheaf $F$ of Chern character $\vi_n$ is slope stable, and there exist unique $(L,I,s)$ such that
$$
	F \ \cong\ \cok (s)\otimes L,
$$
where $I=I_C\otimes T$ is a torsion-free sheaf of Chern character $\vi = (1, 0, -\beta, -m)$, the line bundles $L,T$ have torsion $c_1$,  and $s\colon \cO_X(-n) \rightarrow I$ is non-zero. 
\end{Thm}

We call $(I,s)$ a \emph{Joyce--Song pair} when $0\ne s\in H^0(I(n))$. Joyce and Song \cite[Section 5.4]{JS} studied pairs consisting of a sheaf and a section (of a very positive twist of the sheaf) satisfying a version of Gieseker stability for pairs. We use weak stability conditions instead, but ultimately we prove that in this simplified rank 1 situation, both conditions amount to the same: that $I$ is torsion-free and $s$ is non-zero.

To prove Theorem~\ref{Theorem. part 1}, we start in the large volume limit, where a very similar argument to \cite[Proposition 14.2]{Br} implies that a rank 0 sheaf is slope (semi)stable if and only if it is $\nu_{b,w}$-(semi)stable for any $b \in \mathbb{R}$ and $w \gg 0$. 

So now take a slope semistable sheaf $F$ of Chern character $\vi_n$. It is in the heart $\mathcal{A}(b)$ for any $b \in \mathbb{R}$ and $\nu_{b,w}$-semistable for $w \gg 0$. By Proposition~\ref{prop. locally finite set of walls} the walls of instability for $F$ are all line segments of slope $b_0 := - \frac{n}{2}- \frac{\beta\cdot H}{nH^3}$; see Figure~\ref{figure.walls for class v}. The lowest such wall is tangent to $\partial U$ at $\left(b_0, \frac12b_0^2\right)$. So it makes sense to move down the vertical line $b=b_0$ which intersects all the walls of instability for $F$.
Since
\begin{equation*}
\ch_0(F)\ =\ 0\quad\text{and}\quad \ch_2^{b_0H}(F)\cdot H\ =\ -\beta\cdot H -\frac{n^2H^3}{2} -b_0nH^3\ =\ 0,
\end{equation*}
Conjecture~\ref{BG} gives the Bogomolov--Gieseker inequality \eqref{BGineq} for the stability parameters $(b_0,w)$ where $w\in(w_f -\epsilon,w_f]$. This says that while $F$ is $\nu_{b_0, w}$-semistable,
\begin{equation*}
\ch_3^{b_0H}(F)\ =\ -m + \dfrac{n^3H^3}{6} +b_0H\cdot\left(\beta +\dfrac{n^2H^2}{2}\right) + nH\cdot\frac12b_0^2H^2\ \leq\ \left(\dfrac{w}{3} - \dfrac{b_0^2}{6}\right) nH^3.
\end{equation*}
Rearranging gives
\begin{equation}\label{final}
w\ \geq\ w_f\ :=\ \dfrac{n^2}{4} -\frac{\beta\cdot H}{H^3} - \dfrac{3m}{nH^3}  - \left(\frac{\beta\cdot H}{nH^3}\right)^{\!2}.
\end{equation}
We may assume we chose $n\gg0$ sufficiently large that
\begin{equation*}
w_f\ >\ \frac{b_0^2}{2}\ =\ \frac{n^2}{8} + \frac{\beta\cdot H}{2H^3} + \frac{1}{2}\left(\frac{\beta\cdot H}{nH^3}\right)^{\!2}, 
\end{equation*}
so the point $(b_0, w_f)$ lies inside $U$. Therefore, moving down the line $b= b_0$, there is a point $w_0 \geq w_f$ where $F$ is first destabilised. Our ultimate aim (achieved in Proposition~\ref{prop. exact value of ch2}) will be to show that this point is $w_0=\frac{n^2}{4} + \frac{(\beta\cdot H)^2}{(nH^3)^2}$ where $\{b=b_0\}$ intersects the upper (red) line in Figure~\ref{figure.walls for class v}. This is where $F$ can be destabilised by $\cO(-n)[1]$ in (a rotation of\,) a triangle $\cO(-n)\to I\to F$ made by a Joyce--Song pair for some sheaf $I$ of Chern character $v$.
The next proposition gets us part of the way to this goal.

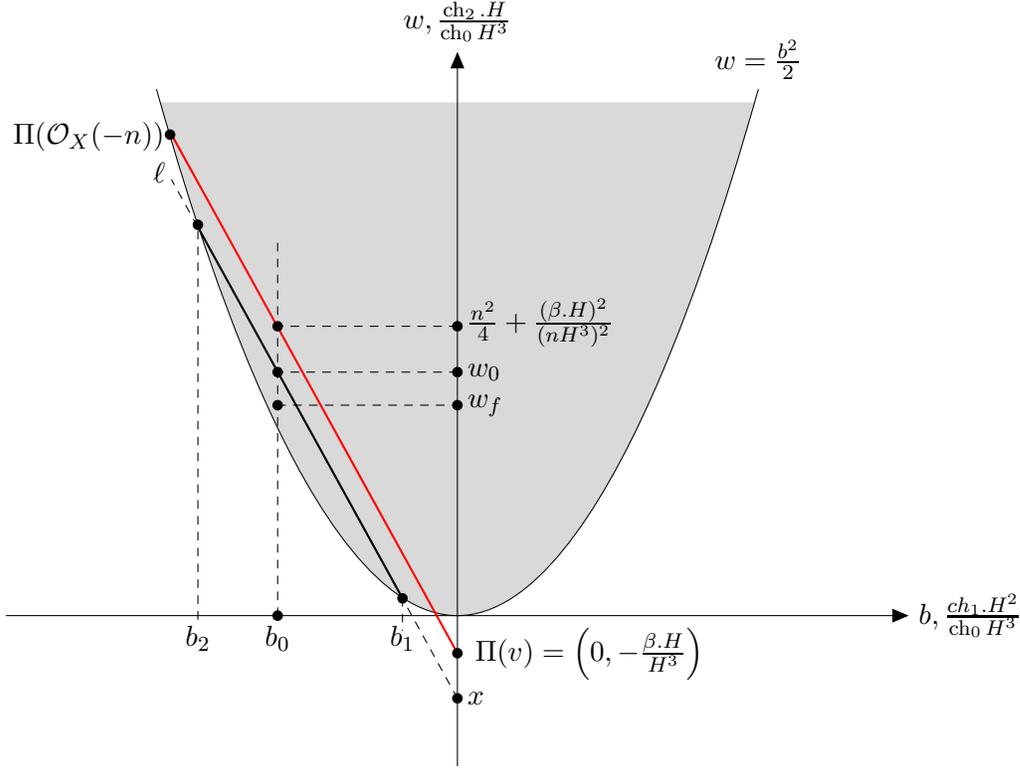
\begin{figure}[h]
	\begin{centering}
		\definecolor{zzttqq}{rgb}{0.27,0.27,0.27}
		\definecolor{qqqqff}{rgb}{0.33,0.33,0.33}
		\definecolor{uququq}{rgb}{0.25,0.25,0.25}
		\definecolor{xdxdff}{rgb}{0.66,0.66,0.66}
		
		\begin{tikzpicture}[line cap=round,line join=round,>=triangle 45,x=1.0cm,y=1.0cm]
		
		\draw[->,color=black] (-6,0) -- (6,0);
		\draw  (6, 0) node [right ] {$b, \frac{ch_1\cdot H^2}{\ch_0H^3}$};
		\fill [fill=gray!30!white] (0,0) parabola (3.95, 6.83) parabola [bend at end] (-3.95, 6.83) parabola [bend at end] (0,0);

		\draw  (0,0) parabola (4,7); 
		\draw  (0,0) parabola (-4,7); 
		\draw  (4 , 7) node [above] {$w=\frac{b^2}{2}$};

		\draw[->,color=black] (0,-2) -- (0,7.5);
		\draw  (0, 7.5) node [above ] {$w, \frac{\ch_2\cdot H}{\ch_0H^3}$};
	
		\draw[dashed,color=black] (-2.39,0) -- (-2.39,5);
		
		\draw[dashed,color=black] (0,3.85) -- (-2.39,3.85);
		
		\draw[dashed,color=black] (0,3.24) -- (-2.39,3.24);
		
		\draw [color=red, thick] (0,-.5) -- (-3.8,6.4);
		\draw [color=black, dashed] (0,-1.1) -- (-3.8,5.8);
		
		\draw [color=black,thick] (-.73, .23) -- (-3.45,  5.17);
		
		\draw [color=black, dashed] (-3.45, -.05) -- (-3.45,  5.15);
		
		\draw[dashed, color=black] (-2.39,2.8) -- (0,2.8);
		\draw [dashed, color=black] (-.73, .23) -- (-.73, -.2);
		\draw (-3.45, 0) node [below]{$b_2$};
		\draw (-.73, 0) node [below]{$b_1$};
		
		\draw  (-4.9, 6.7) node [below ] {$\Pi(\mathcal{O}_X(-n))$};
		\draw  (0.1, -.5) node [right ] {$\Pi(\vi) = \left(0,-\frac{\beta\cdot H}{H^3}\right)$};
		\draw  (0, 3.9) node [right ] {$\frac{n^2}{4} + \frac{(\beta\cdot H)^2}{(nH^3)^2}$ };
		\draw  (0, 3.24) node [right ] {$w_0$};
		\draw  (0, 2.8) node [right ] {$w_f$};
		\draw  (-2.39, 0) node [below ] {$b_0$};
		\draw  (-3.75,5.9) node [left ] {$\ell$};
		\draw  (0, -1.1) node [right] {$x$};
		
		\begin{scriptsize}
		\fill (-2.39, 2.8) circle (2pt);
		\fill (0, 2.8) circle (2pt);
				
		
		\fill (0, -.5) circle (2pt);
		\fill (0, 3.85) circle (2pt);
		
		\fill (-3.82,6.4) circle (2pt);
		\fill (-2.39,0) circle (2pt);
		\fill (-2.39,3.85) circle (2pt);

		\fill (-2.39,3.24) circle (2pt);
		\fill (0,3.24) circle (2pt);
		
		\fill (0,-1.1) circle (2pt);
    	\fill (-.73, .23) circle (2pt);
		\fill (-3.45,  5.2) circle (2pt);

		\end{scriptsize}
		
		\end{tikzpicture}
		
		\caption{Walls for objects of class $\vi_n$}
		
		\label{figure.walls for class v}
		
	\end{centering}
\end{figure}

\begin{Prop}\label{prop.the first wall}
    The wall that bounds the large volume limit chamber $w \gg 0$ for $F$ has slope $b_0=-\frac n2-\frac{\beta\cdot H}{nH^3}$ and passes through the point $(b_0, w_0)$, where 
    $$
    w_f\ \le\ w_0\ \le\ \frac{n^2}{4} + \left(\frac{\beta\cdot H}{nH^3}\right)^2. 
    $$
     On this wall there is a destabilising sequence $F_1 \hookrightarrow F \twoheadrightarrow F_2$ in $\mathcal{A}(b_0)$, where $F_2$ is a 2-term complex with $\dim\supp\cH^0(F_2)\le1$ and $F_1$ is a rank $1$ torsion-free sheaf with 
\begin{equation}\label{dagger}
    \ch_1(F_1)\cdot H^2\ =\ 0 \quad\text{and}\quad -2 \beta\cdot H\ \leq\ \ch_2(F_1)\cdot H\ \leq\ -\beta\cdot H.
\end{equation}
\end{Prop}

\begin{proof}
    By Conjecture~\ref{BG} and \eqref{final}, $F$ is $\nu_{b_0,w_0}$-destabilised by a sequence $F_1 \hookrightarrow F \twoheadrightarrow F_2$ in $\mathcal{A}(b_0)$ for some $w_0 \geq w_f$. By Proposition~\ref{prop. locally finite set of walls} the corresponding wall is $\ell\cap U$ for $\ell$ the line pictured in Figure~\ref{figure.walls for class v} with equation 
    $w = b_0 b + x$, where
    \begin{equation}\label{in. for x}
    x\ =\ -b_0^2 +w_0\ \geq\ -b_0^2 + w_f\ =\ -\frac{2\beta\cdot H}{H^3} - \frac{3m}{nH^3} -2\left(\frac{\beta\cdot H}{nH^3}\right)^{\!2}.
    \end{equation}
    Let $b_2 < b_1 $ be the values of $b$ at the intersection points of $\ell$ and the boundary $\partial U=\big\{w = \frac{1}{2}b^2\big\}$ of the space of weak stability conditions $U$, 
    \begin{equation*}
      b_1\ =\ b_0 + \sqrt{b_0^2+2x}\,, \quad b_2\ =\ b_0 - \sqrt{b_0^2+2x}\,. 
    \end{equation*}
    We claim that 
    \begin{equation}\label{in. lower bound for b1}
    b_1\ >\ -\frac{1}{2H^3} \quad\text{and}\quad b_2\ <\ -n + \frac{1}{2H^3}\,.
    \end{equation}
The first is equivalent to $b_0^2 +2 x > \left(-b_0 - \frac{1}{2H^3} \right)^2$, and thus to  
    \begin{equation*}
	  x\ >\ \frac{b_0}{2H^3} + \frac{1}{8(H^3)^2}\ =\ -\frac{n}{4H^3} - \frac{\beta\cdot H}{2n(H^3)^2} + \frac{1}{8(H^3)^2}\,.
	\end{equation*}
The second  is equivalent to $\left(b_0 +n -\frac{1}{2H^3} \right)^2 < b_0^2 + 2x$, and therefore to 
    \begin{equation*}
    x\ >\ \left(n- \frac{1}{2H^3} \right)\left(b_0 + \frac{1}{2} \left(n- \frac{1}{2H^3} \right) \right)\ =\ -\left(n- \frac{1}{2H^3} \right) \left( \frac{\beta\cdot H}{nH^3} + \frac{1}{4H^3}\right). 
    \end{equation*}
Both of these follow from \eqref{in. for x} for $n\gg0$.

By Proposition~\ref{prop. locally finite set of walls} 
there is a short exact sequence $F_1 \hookrightarrow F \twoheadrightarrow F_2$ in $\cA(b_0)$ which strictly destabilises $F$ below the wall. Taking cohomology gives the long exact sequence of coherent sheaves
   \begin{equation}\label{long exact}
   0 \To \cH^{-1}(F_2) \To \cH^0(F_1) \To F \rightarrow \cH^0(F_2) \To 0.
   \end{equation}
   	In particular, the destabilising subobject $F_1$ is a coherent sheaf. If it had rank 0, then its slope $\nu_{b,w}(F_1)$, see \eqref{noo}, would be constant throughout $U$, like that of $F$, so we would not have a wall. Thus $\ch_0(F_1)>0$, so \eqref{long exact} gives
$$
   \ch_0(\cH^{-1}(F_2))\ =\ \ch_0(F_1)\ >\ 0.
$$
Since $\rk\,F_2=-\rk\,F_1\ne0$, Proposition~\ref{prop. locally finite set of walls} shows that $\Pi(F_1)$ and $\Pi(F_2	)$ lie on the line $\ell$. All along $\ell\cap U$ --- \textit{i.e.}\ for $b \in (b_2, b_1)$ --- the objects $F_1$ and $F_2$ lie in the heart $\mathcal{A}(b)$ and (semi)destabilise $F$. Therefore, by the definition \eqref{Abdef} of $\cA(b)$,
   \begin{equation}\label{two conditions}
   \mu_H^{+}(\cH^{-1}(F_2))\ \leq\ b_2\ <\ -n + \frac{1}{2H^3} 
   \quad\text{and}\quad \mu_H^{-}(F_1)\ \geq\ b_1\ >\ -\frac1{2H^3}\,.
   \end{equation}  
   Thus intersecting the identity $\ch_1(F)-\ch_1(\cH^0(F_2))=\ch_1(F_1)-\ch_1(\cH^{-1}(F_2))$ with $H^2$ and dividing by $\ch_0(F_1)H^3$ gives
   \begin{align}\label{in.1}
   \frac n{\ch_0(F_1)}-\dfrac{\ch_1(\cH^0(F_2))\cdot H^2}{\ch_0(F_1)H^3}\ &=\ \mu\_H(F_1) -\mu\_H(\cH^{-1}(F_2)) \\ \nonumber
   &\ge\ \mu_H^{-}(F_1) - \mu_H^{+}(\cH^{-1}(F_2))\ \geq\ b_1 - b_2\ >\ n- \frac{1}{H^3}\,,
   \end{align}
  with the last inequality following from \eqref{in. lower bound for b1}.
    Now $\ch_1(\cH^0(F_2))\cdot H^2 \geq 0$ since $\cH^0(F_2)$ has rank 0. So \eqref{in.1} can only hold for $n>0$ if $\ch_0(F_1) =1$ and $\ch_1(\cH^{0}(F_2))\cdot H^2 =0$. In particular, $\cH^{0}(F_2)$ is supported in dimension at most 1. Plugging $\rk\,F_1=1$ back into \eqref{two conditions} gives 
	\begin{equation}\label{in. lower bound}
\frac{\ch_1(F_1)\cdot H^2}{H^3}\ =\ \mu\_H(F_1)\ >\ -\frac{1}{2H^3}, \quad\text{so}\quad \mu\_H(F_1)\ \geq\ 0.
	\end{equation}
	Similarly, plugging $\rk\,\cH^{-1}(F_2))=1$ into \eqref{two conditions} gives
     $$
\frac{\ch_1(\cH^{-1}(F_2))\cdot H^2}{H^3}\ =\	\mu\_H(\cH^{-1}(F_2))\ <\ -n + \frac{1}{2H^3}, \quad\text{so}\quad \mu\_H(\cH^{-1}(F_2))\ \leq\ -n.
     $$
    On the other hand, \eqref{in.1} gives $n = \mu\_H(F_1) - \mu\_H(\cH^{-1}(F_2))$, so  in fact
    $$
    \mu\_H(F_1)\ =\ 0 \quad\text{and}\quad \mu\_H(\cH^{-1}(F_2))\ =\ -n\ =\ \mu\_H(F_2).
    $$
We use this to show that $F_1$ is torsion-free. Suppose for a contradiction that its torsion subsheaf $T$ is non-zero. If $\ch_1(T)\cdot H^2>0$, then since $\mu\_H(F_1)=0$, we find $\mu\_H(F_1/T)\le-\frac{1}{H^3}$. But $\mu_H^-(F_1)\le\mu\_H(F_1/T)$, so this contradicts \eqref{two conditions}. Therefore, $\ch_1(T)\cdot H^2=0$, and $T$ is supported in codimension at least 2. But then $\nu_{b_0,w_0}(T)=+\infty$, so $T$ strictly destabilises $F_1$,  contradicting its $\nu_{b_0,w_0}$-semistability.\medskip
    
    To finish the proof, we consider the projected point $\Pi(F_2) = \left( -n \,,\, -\frac{\ch_2(F_2)\cdot H}{H^3}\right)$. Since it lies on the line $\ell=\{w = b_0b+x\}$,
    \begin{equation*}
    -\frac{\ch_2(F_2)\cdot H}{H^3}\ =\ n \left(\frac{n}{2} + \frac{\beta\cdot H}{nH^3} \right) +x.   
    \end{equation*}
    Since $F_2$ is $\nu_{b_0,w_0}$-semistable, it satisfies the classical Bogomolov--Gieseker inequality \eqref{discr},
        \begin{equation}\label{in. upper bound for x}
      -\frac{\ch_2(F_2)\cdot H}{H^3}\ \leq\ \frac{n^2}{2}, \quad\text{so that}\quad x\ \leq\ -\frac{\beta\cdot H}{H^3}\,. 
    \end{equation}
  This proves that the line $\ell$ cannot be above the red uppermost line in Figure~\ref{figure.walls for class v}; \textit{i.e.}\  
  \begin{equation*}
     w_0\ =\ b_0^2 +x\ \leq\ \left(-\frac{n}{2} - \frac{\beta\cdot H}{nH^3} \right)^2- \frac{\beta\cdot H}{H^3}\ =\ \frac{n^2}4+\left(\frac{\beta\cdot H}{nH^3}\right)^{\!2}, 
  \end{equation*}
    as claimed in the proposition. Moreover, $\Pi(F_1) = \left( 0, \frac{\ch_2(F_1)\cdot H}{H^3}\right)\in\ell$ gives
\begin{equation}\label{x}
x\ =\ \frac{\ch_2(F_1)\cdot H}{H^3}\,,
\end{equation}
so the inequalities \eqref{in. upper bound for x} and \eqref{in. for x} imply 
$$
    -\frac{\beta\cdot H}{H^3}\ \geq\ \frac{\ch_2(F_1)\cdot H}{H^3}\ \geq\ -\frac{2\beta\cdot H}{H^3} - \frac{3m}{nH^3} -2\left(\frac{\beta\cdot H}{nH^3}\right)^2.
$$
For $n\gg0$ the second inequality becomes $\frac{\ch_2(F_1)\cdot H}{H^3}\ge-\frac{2\beta\cdot H}{H^3}$.
\end{proof}

\begin{Rem}
If $\beta\cdot H<0$, then \eqref{dagger} gives a contradiction, showing that no such $F$ exists and $M_{X,H}(\vi_n)$ is empty. But then $I_m(X,\beta)$ is also empty since there are no sheaves $I_C\otimes T$ with $[C]=\beta$ when $\beta\cdot H<0$. Thus Theorem~\ref{theorem.1} holds in this case, so from now on we will assume $\beta\cdot H\ge0$.
\end{Rem}

We now borrow some results from \cite[Section 8]{FT}. There we explain why it is profitable to work on the vertical line $\{b = b'\}\cap U$, where $b':=-\frac{1}{H^3}$; see Figure~\ref{figure. the object F1}. In particular, a simple numerical argument (precisely analogous to the argument that rank 1 sheaves can only be slope destabilised by their rank 0 torsion subsheaves) shows that rank 1 objects with $\ch_1\!\cdot H^2=0$ are $\nu_{b',w}$-semistable at a point of this vertical line if and only if they are \emph{$\nu_{b',w}$-stable everywhere on the line}.

\begin{Lem}[\textit{cf.}~\protect{\cite[Section 8]{FT}}]\label{lem. no wall for F1} 
  Take an object $E \in \mathcal{A}(b')$ of rank $1$ with $\ch_1(E)\cdot H^2 = 0$. If\, $E$ is $\nu_{b',w}$-semistable for some $w > \frac12(b')^2$, then it is $\nu_{b',w}$-stable for all $w > \frac12(b')^2$.
\end{Lem} 

By \eqref{in. lower bound for b1}, we have $b_1>-\frac1{2H^3}>b'$, so Figure~\ref{figure. the object F1} shows that our wall of instability $\ell$ for $F$ intersects $\{b = b'\}$ at an interior point of $U$. Therefore, the $\nu_{b,w}$-semistability of $F_1$ along the line segment $\ell \cap U$ and Lemma~\ref{lem. no wall for F1} show that $F_1$ is $\nu_{b',w}$-semistable for all $w > \frac12(b')^2$.

Let $\ell_1$ be the line connecting $\Pi(F_1) = \left(0 , \frac{\ch_2(F_1)\cdot H}{H^3} \right)$ to the point $\big(b',\frac12(b')^2\big)$ where $\{b=b'\}$ intersects $\partial U$. By Proposition~\ref{prop. locally finite set of walls} the $w \downarrow \frac{1}{2(H^3)^2}$ limit of Lemma~\ref{lem. no wall for F1} shows that $F_1$ is $\nu_{b,w}$-semistable for any $(b,w) \in \ell_1 \cap U$. 

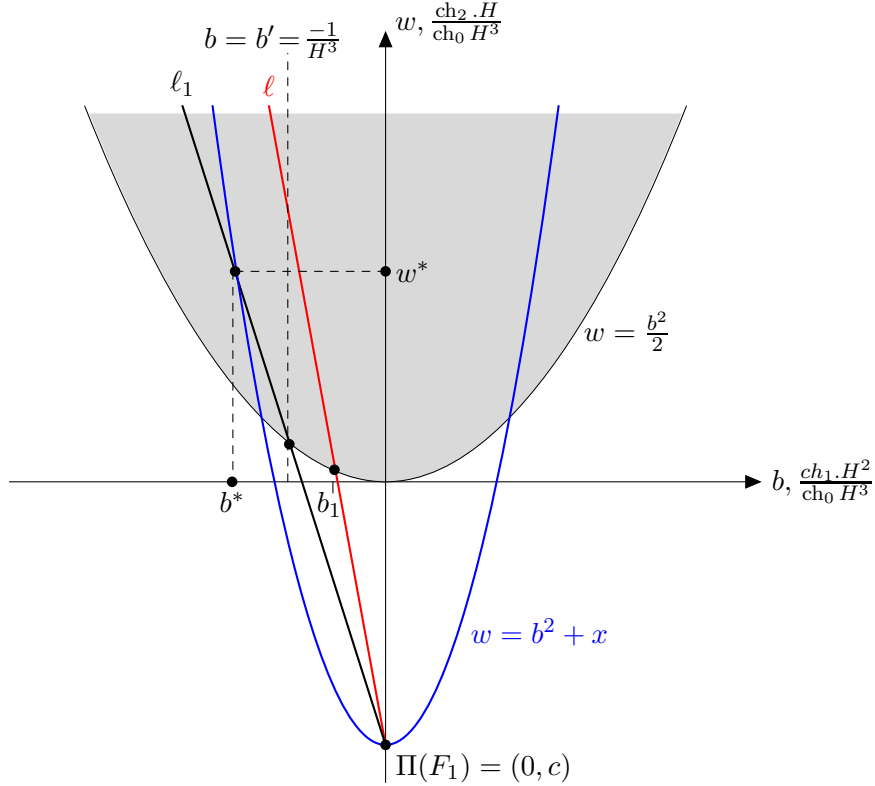
\begin{figure}[h]
	\begin{centering}
		\definecolor{zzttqq}{rgb}{0.27,0.27,0.27}
		\definecolor{qqqqff}{rgb}{0.33,0.33,0.33}
		\definecolor{uququq}{rgb}{0.25,0.25,0.25}
		\definecolor{xdxdff}{rgb}{0.66,0.66,0.66}
		
		\begin{tikzpicture}[line cap=round,line join=round,>=triangle 45,x=1.0cm,y=1.0cm]
		
		\draw[->,color=black] (-5,0) -- (5,0);
		\draw  (5, 0) node [right ] {$b, \frac{ch_1\cdot H^2}{\ch_0H^3}$};
		\fill [fill=gray!30!white] (0,0) parabola (3.94, 4.9) parabola [bend at end] (-3.94, 4.9) parabola [bend at end] (0,0);

		\draw  (0,0) parabola (4,5); 
		\draw  (0,0) parabola (-4,5); 
%
%
		\draw[->,color=black] (0,-4) -- (0,6);
		\draw  (.85, 5.7) node [above ] {$w, \frac{\ch_2\cdot H}{\ch_0H^3}$};
		\draw[color=red, thick] (0,-3.5) -- (-1.55,5);
		\draw  (-1.55,5) node [above, color=red] {$\ell$};
			
		\draw  (-.75,0) node [below] {$b_1$};
		\draw (-.7, 0) -- (-.7,-.15);
		
		\draw[color=black, thick] (0,-3.5) -- (-2.7,5);
		\draw  (-2.7,5) node [above] {$\ell_1$};
		\draw[color=black, dashed] (-1.3,0) -- (-1.3,5.7);
		\draw  (-1.5, 5.5) node [above ] {$b = b'\!=\!\frac{-1}{H^3}$};

		\draw [color=blue, thick]  (0,-3.5) parabola (-2.3,5);			
		\draw [color=blue, thick]  (0,-3.5) parabola (2.3,5);
%
%
%
%
%
		\draw[dashed, color=black] (-2, 2.8) -- (0, 2.8);
		\draw[dashed, color=black] (-2.03, 2.8) -- (-2.03, 0);
%
		\draw  (0, 2.8) node [right ] {$w^*$};
		\draw  (-2.04, 0) node [below ] {$b^*$};
        \draw  (2.5, 2) node [right] {$w = \frac{b^2}{2}$};
		\draw  (1, -2) node [right, color=blue ] {$w =b^2 + x$};
		\draw  (0, -3.8) node [right] {$\Pi(F_1)=(0,c)$};
		
		\begin{scriptsize}
%
%
%
		\fill (0,-3.5) circle (2pt);
		\fill (-.68,.16) circle (2pt);
		\fill (-2.04, 0) circle (2pt);
		\fill (0, 2.8) circle (2pt);
		\fill (-2,2.8) circle (2pt);
		\fill (-1.28,.5) circle (2pt);
%
%
%

		\end{scriptsize}
		
		\end{tikzpicture}
		
		\caption{Walls for objects of class $\vi$}
		
		\label{figure. the object F1}
		
	\end{centering}
\end{figure}

So by Conjecture~\ref{BG}\eqref{BG-2} we can apply the Bogomolov--Gieseker inequality \eqref{BGineq} to $F_1$, so long as we can find a point of $\ell_1 \cap U$ satisfying $\ch_2^{bH}(F_1)\cdot H = \left(w- \frac{1}{2}b^2\right)\ch_0(F_1)H^3$, \textit{i.e.}\
$$
\frac{\ch_2(F_1)\cdot H}{H^3} + \frac{b^2}{2}\ =\ w -\frac{b^2}{2}\,.
$$
This gives the lower parabola in blue in Figure~\ref{figure. the object F1}. It intersects $\ell_1$ at $(b^*, w^*)$, where 
\begin{equation}\label{b*w*}
    b^*\ =\ \ch_2(F_1)\cdot H - \frac{1}{2H^3}\,, \quad w^*\ =\ (b^*)^2 + \frac{\ch_2(F_1)\cdot H}{H^3}\,. 
\end{equation}
Since 
\begin{equation*}
    w^* - \frac{(b^*)^2}{2}\ =\   \frac{1}{2}\left(\ch_2(F_1)\cdot H - \frac{1}{2H^3}\right)^{\!2} + \frac{\ch_2(F_1)\cdot H}{H^3}\ >\ 0,
\end{equation*}
the point $(b^*,w^*)$ is in the interior of $U$.
As in \cite[Proposition 8.3]{FT}, Conjecture~\ref{BG}\eqref{BG-2} then gives the following.
 
\begin{Prop}\label{prop. upper bound for ch_3(F_1)}
	If\, $\beta\cdot H >0$, the destabilising sheaf\, $F_1$ satisfies
	\begin{equation*}
	\ch_3(F_1)\ \leq\ \frac{2}{3}\ch_2(F_1)\cdot H\left(\ch_2(F_1)\cdot H- \frac{1}{2H^3}  \right). \vspace{-8mm}
	\end{equation*}
	$\hfill\square$
\end{Prop}\vspace{2mm}

\noindent A similar argument given in \cite[Proposition 8.4]{FT} gives a similar inequality for $F_2(n)$.

\begin{Prop} \label{prop. upper bound for ch_3(F_2)}
    If\, $\beta\cdot H > 0$, the destabilising object $F_2$ satisfies
    \begin{equation*}
      \ch_3(F_2(n))\ \leq\ \frac{2}{3}\ch_2(F_2(n))\cdot H\left(\ch_2(F_2(n))\cdot H + \frac{1}{2H^3}  \right). \vspace{-8mm}
	\end{equation*}
	$\hfill\square$
\end{Prop}\vspace{2mm}

\begin{Prop}\label{prop. exact value of ch2}
     We have $\ch_2(F_1)\cdot H = -\beta\cdot H$ and $\ch_2(F_2(n))\cdot H=0$.
\end{Prop}
\begin{proof}  Set $c := \ch_2(F_1)\cdot H$.
    If $\beta\cdot H = 0$, then \eqref{dagger} gives $c = -\beta\cdot H = 0$. So we now assume $\beta\cdot H >0$.        Using $\ch_0(F_1) =1$, $\ch_1(F_1)\cdot H^2 =0$ and the exact triangle $F_1\to F \to F_2$, we compute 
    \begin{align}\label{chi}
      \ch_1(F_2(n))\cdot H^2\ =\ 0, \quad \ch_2(F_2(n))\cdot H\ =&\ -\beta\cdot H -c \\
      \text{and} \quad \ch_3(F_2(n))\ =&\ -m- \ch_3(F_1) - n(\beta\cdot H +c).
    \end{align}
Thus Proposition~\ref{prop. upper bound for ch_3(F_2)} becomes
    \begin{equation}\label{mbeta}
     -m- \ch_3(F_1) - n(\beta\cdot H +c)\ \leq\ \frac{2}{3}(\beta\cdot H +c) \left(\beta\cdot H +c \ - \frac{1}{2H^3}  \right),
    \end{equation}
while  Proposition~\ref{prop. upper bound for ch_3(F_1)} says
    \begin{equation}\label{in. ch3 bound}
      \ch_3(F_1)\ \leq\ \frac{2}{3}c \left( c- \frac{1}{2H^3} \right). 
    \end{equation}
    Combining the two gives
    \begin{equation}\label{in. bounds}
-m-\frac{2}{3}c \left( c- \frac{1}{2H^3} \right)- n(\beta\cdot H +c)\ \leq\ \frac{2}{3}(\beta\cdot H +c) \left(\beta\cdot H +c \ - \frac{1}{2H^3}  \right).
    \end{equation}
By  Proposition~\ref{prop.the first wall} we have $c\in[-2\beta\cdot H,-\beta\cdot H]$, so the right-hand side of \eqref{in. bounds} is bounded while on the left-hand side, $n$ appears multiplied by $-(\beta\cdot H+c)\ge0$. This gives a contradiction for $n\gg0$ unless $-(\beta\cdot H+c)=0$.
\end{proof}

\begin{Lem}\label{lem. exact value of ch1}
The sheaf\, $\cH^0(F_2)$ is supported in dimension 0, and
$$
\ch_1(\cH^{-1}(F_2))\ =\ -nH \quad\text{in\ } H^2(X,\Q).
$$
	\end{Lem}
	
\begin{proof}
Proposition~\ref{prop.the first wall} and the long exact sequence \eqref{long exact} show that $\cH^{-1}(F_2)$ is a torsion-free sheaf of rank 1. Therefore, it is $\mu\_H$-semistable, and the classical Bogomolov inequality applies, 
	\begin{equation}\label{condition 2}
	\ch_1(\cH^{-1}(F_2))^2\cdot H -2\ch_2(\cH^{-1}(F_2))\cdot H\ \geq\ 0. 
	\end{equation}
By \eqref{long exact} again, $\ch_i(\cH^{-1}(F_2))=\ch_i(F_1)-\ch_i(F)+\ch_i(\cH^0(F_2))$. 	Take $i=2$ and intersect with $H$; then Proposition~\ref{prop. exact value of ch2} gives
	\begin{equation}\label{ch2}
	\ch_2(\cH^{-1}(F_2))\cdot H\ =\
	-\beta\cdot H + \beta\cdot H +\frac{n^2H^3}{2} + \ch_2(\cH^0(F_2))\cdot H\ =\ \frac{n^2H^3}{2} + \ch_2(\cH^0(F_2))\cdot H.
	\end{equation}
Take $i=1$ and intersect with $H^2$; then Proposition~\ref{prop.the first wall} kills the first and third terms, so
\begin{equation}\label{in. ch1}
\ch_1(\cH^{-1}(F_2))\cdot H^2\ =\ -nH^3.
\end{equation}
Therefore, by the Hodge index theorem,
\begin{equation}\label{ch11}
n^2H^3\ =\ \frac{\left(\!\ch_1(\cH^{-1}(F_2))\cdot H^2\right)^2}{H^3}\ \ge\ \ch_1(\cH^{-1}(F_2))^2\cdot H,
\end{equation}
with equality if and only if $\ch_1(\cH^{-1}(F_2))$ is a multiple of $H$ in $H^2(X,\Q)$.

Combining \eqref{condition 2}, \eqref{ch2} and \eqref{ch11} gives
\begin{equation}\label{<>0}
\ch_2(\cH^0(F_2))\cdot H\ \le\ 0.
\end{equation}
But $\dim\supp\cH^0(F_2)\le1$ by Proposition~\ref{prop.the first wall}, so this shows  $\dim\supp\cH^0(F_2)=0$, and (\ref{<>0}) and~(\ref{ch11}) are equalities. Thus $\ch_1(\cH^{-1}(F_2))$ is a multiple of $H$ in $H^2(X,\Q)$, and by \eqref{in. ch1} that multiple is $nH$.
	\end{proof}

\begin{Lem}\label{lem. ch3}
We have $\ch_3(F_2)\le\frac16n^3H^3$. 
\end{Lem}

\begin{proof}
Lemma~8.2 of \cite{FT} implies that $F_2(n)\in \cA(-b')$ is $\nu_{-b', w}$-semistable for $b' = -\frac{1}{H^3}$ and $w \gg 0$. Therefore, by \cite[Lemma 5.1.3(b)]{BMT} the shifted derived dual $F_2(n)^{\vee}[1]$ lies in an exact triangle 
  \begin{equation}\label{Q}
     E\ \Into\, F_2(n)^{\vee}[1]\, \To\hspace{-5.5mm}\To\, Q[-1],
  \end{equation}
with $Q$ a 0-dimensional sheaf and $E$ a $\nu\_{b',w}$-semistable object of $\cA(b')$ for $w \gg 0$. Since $\rk E=1$, it is a torsion-free sheaf by \cite[Lemma 2.7]{BMS}. Moreover, Lemma~\ref{lem. exact value of ch1} and  Proposition~\ref{prop. exact value of ch2} give 
$$
\ch_1(E)\ =\ 0\ \,\text{in\ } H^2(X,\Q), \quad \ch_2(E)\cdot H\ =\ 0. 
$$
Hence $\ch_3(E) \leq 0$, which by \eqref{Q} gives $\ch_3(F_2(n)) \leq 0$. 
\end{proof}

We are finally ready to identify the destabilising sequence for $F$.
By Lemma~\ref{lem. exact value of ch1} there is a  subscheme $Z\subset X$ of dimension at most 1 such that
\begin{equation}\label{TC}
\cH^{-1}(F_2)\ \cong\ L(-n)\otimes I_Z
\end{equation}
for some line bundle $L$ with $c_1(L)=0\in H^2(X,\Q)$. By \eqref{ch2} we find $\ch_2(L(-n)\otimes I_Z)\cdot H=\frac12n^2H^3=\ch_2(L(-n))\cdot H$, so in fact $Z$ is 0-dimensional. If it were non-empty, then $\nu_{b_0,w}(\cO_Z)=+\infty$, so combining the $\cA(b_0)$-short exact sequences
$$
\cO_Z\Into\cH^{-1}(F_2)[1]\Onto L(-n)[1] \quad\text{and}\quad \cH^{-1}(F_2)[1]\Into F_2\Onto\cH^0(F_2)
$$
gives the destabilising subobject $\cO_Z\into F_2$. This contradicts the $\nu_{b_0,w_0}$-semistability of $F_2$, so in fact $Z=\emptyset$. Therefore, $\cH^{-1}(F_2)\cong L(-n)$, which has $\ch_3=-\frac16n^3H^3$.

Since $\ch_3(F_2)\le\frac16n^3H^3$ by Lemma~\ref{lem. ch3}, this gives $\ch_3(\cH^0(F_2))\le0$. But by Lemma~\ref{lem. exact value of ch1} $\cH^0(F_2)$ is 0-dimensional, so it vanishes and
$$
F_2\ \cong\ L(-n)[1].
$$
Thus our destabilising sequence in $\cA(b_0)$ is
\begin{equation}\label{at last}
I\Into F\Onto L(-n)[1]
\end{equation}
for some $L\in\Pic\_0(X)$ and rank 1 torsion-free sheaf $I:=F_1\in I_m(X,\beta)\times\Pic\_0(X)$ of Chern character $\vi = (1, 0, -\beta , -m)$. Therefore, $F\otimes L^*$ is (the cokernel of) a Joyce--Song pair $\cO(-n)\to I\otimes L^*$.

\subsection{Uniqueness} To finish the proof of Theorem~\ref{Theorem. part 1}, we should prove the uniqueness of the sequence \eqref{at last} and the slope stability of $F$.

By \cite[Lemma 2.7(c)]{BMS} the slope semistable sheaf $I$ is $\nu_{b', w}$-semistable for $w \gg 0$. Let $\ell'$ be the red line segment in Figure~\ref{figure.walls for class v} which passes through $\Pi(\cO_X(-n))$ and $\Pi(v)$. By Lemma~\ref{lem. no wall for F1} --- and the fact noted there that $\ell'\cap U$ intersects $\{b=b'\}$ --- it is strictly stable all along $\ell'\cap U$. In particular, it is $\nu_{b_0,w_0}$-stable. The same is true of $L(-n)[1]$, by \cite[Corollary 3.11(a)]{BMS}. That is,
\begin{equation}\label{staybul}
I\text{ and }L(-n)[1]\text{ are }\nu_{b_0, w_0}\text{-stable of the same phase.}
\end{equation}
As we move below the wall $\ell$ to $w=w_0-\epsilon$, they remain stable for $0<\epsilon\ll1$ but 
    \begin{equation*}
    \nu\_{b_0, w_0-\epsilon}(I)\ >\ \nu\_{b_0, w_0-\epsilon}\left(L(-n)[1]\right)
    \end{equation*}
   by an elementary calculation with \eqref{noo}. Therefore, \eqref{at last} is the Harder--Narasimhan filtration of $F$ with respect to $\nu_{b_0,w_0-\epsilon}$. The uniqueness of the Harder--Narasimhan filtration gives the uniqueness of $I$ and $L$.

    \subsection{Slope stability} It remains to prove that $F$ is not strictly slope semistable in the sense of \eqref{nuslope}. Suppose $F \twoheadrightarrow F'$ is a proper quotient sheaf with $\nu\_H(F') = \nu\_H(F)$.    
    
     Since $\rk\,F'=0=\rk\,F$, the formula \eqref{noo} gives
    $$
    \nu\_{b,w}(F')\ =\ \nu\_H(F')\ =\ \nu\_H(F)\ =\ \nu\_{b,w}(F)
    $$
    for all $(b,w) \in U$. Since all torsion sheaves are in $\cA(b_0)$, $F'$ is a quotient of $F$ in the abelian category $\cA(b_0)$, and any quotient of $F'$ in $\mathcal{A}(b_0)$ is also a quotient of $F$. Therefore, $F'$ is also $\nu_{b_0, w_0}$-semistable.
    Since $I$ is $\nu_{b_0,w_0}$-stable, the composition
$$
I \Into F \Onto F'
$$
in $\cA(b_0)$ must be either zero or injective. And it cannot be zero because this would give a surjection $L(-n)[1] \twoheadrightarrow F'$ in $\mathcal{A}(b_0)$, contradicting the $\nu_{b_0,w_0}$-stability \eqref{staybul} of $L(-n)[1]$.
    
   So it is injective. Let $C$ denote its cokernel in $\cA(b_0)$, sitting in a commutative diagram
 	    \begin{equation*}
	\xymatrix@R=16pt{
		\,I\,\ar@{^{(}->}[r]\ar@{=}[d] & F \ar@{->>}[r]\ar@{->>}[d]&L(-n)[1]\ar@{->>}[d]\\
		\,I\,\ar@{^{(}->}[r] &F'\ar@{->>}[r]&\,C\rlap{.}}
	\end{equation*}
         	Since $F'$ and $I$ are $\nu_{b_0,w_0}$-semistable of the same phase, $C$ is also $\nu_{b_0,w_0}$-semistable.
Therefore, the right-hand surjection contradicts the $\nu_{b_0,w_0}$-stability \eqref{staybul} of $L(-n)[1]$.

\section{Proof of the main theorem}
In this section we prove the rest of Theorem~\ref{theorem.1}. Let $\js_n(\vi)$ be the moduli space of pairs $(I, s)$, where $I=I_C\otimes T$ is a torsion-free sheaf of Chern character $\vi = (1, 0, -\beta, -m)$ and $s \colon \mathcal{O}_X(-n) \rightarrow I$ is a non-zero section. Since there are no strictly semistable sheaves of rank 1, this is a special case of the projective moduli space constructed in \cite[Section 12.1]{JS}. For $n\gg0$ it is a projective bundle over $I_m(X,\beta)\times\Pic\_0(X)$ with fibre $\PP\big(H^0(I(n))\big)$; see \cite[Lemma 3.2]{GST} for instance. For any such pair $(I,s)$ the cokernel $\cok(s)$ is a sheaf of Chern character $\vi_n$. 

\begin{Prop}\label{prop.converse}
   Take a pair $(I, s) \in \js_n(\vi)$ for $n\gg0$. Then $\cok(s)$ is slope stable. 
\end{Prop}
\begin{proof}
By the same argument as in \eqref{staybul}, $I$ and $\cO_X(-n)[1]$ are $\nu_{b_0,w_0}$-stable of the same phase, where $w_0=\frac{n^2}4+\frac{(\beta\cdot H)^2}{(nH^3)^2}$. Hence the exact sequence
 	 	\begin{equation}\label{in. HN}
 	I \Into \cok(s) \Onto \cO_X(-n)[1]
 	\end{equation} 
 	 in $\cA(b_0)$ shows that $\cok(s)$ is also $\nu_{b_0,w_0}$-semistable. Therefore, it is also slope semistable: any quotient sheaf $\cok(s) \twoheadrightarrow F'$ is torsion, so lies in $\mathcal{A}(b_0)$ and satisfies
 	  	\begin{equation*}
  \nu\_{b_0,w_0}\left(\!\cok(s)\right)\ =\ \nu\_H\left(\!\cok(s)\right)\ \leq \ \nu\_H(F')\ =\ \nu\_{b_0,w_0}(F').
 	\end{equation*}
As we just proved in Theorem~\ref{Theorem. part 1}, this implies that $\cok(s)$ is actually slope stable.
\end{proof}
\begin{proof}[Proof of Theorem~\ref{theorem.1}]
   By Theorem~\ref{Theorem. part 1} and Proposition~\ref{prop.converse} we have now proved that any slope or Gieseker semistable sheaf $F$ of Chern character $\vi_n$ is slope (and so Gieseker) stable, and we have established a bijection
	\begin{align}\label{morph}
	\js_n(\vi)\times\Pic\_0(X)\ \To&\  M_{X,H}(\vi_n), \\
	\left((I, s),\,L \right)\ \longmapsto&\ \cok (s) \otimes L.  \nonumber
	\end{align}
Next we make the arrow into a morphism. By \cite[Theorem 4.11]{Le Potier} the product $\js_n(\vi)\times X$ carries a universal Joyce--Song pair. Tensoring with (the pull back of) a Poincar\'e sheaf on $X \times \Pic\_0(X)$ gives a universal complex on $\js_n(\vi)\times\Pic\_0(X)\times X$. Its cokernel is a flat family of sheaves over $\js_n(\vi)\times\Pic\_0(X)$ whose closed fibres are slope and Gieseker stable sheaves of Chern character $\vi_n$. It is therefore classified by a map to the moduli space $M_{X,H}(\vi_n)$, which gives \eqref{morph}. \medskip
    
We are left with finding the inverse morphism. Start with $F\in M_{X,H}(\vi_n)$. By Theorem~\ref{Theorem. part 1} we find a unique $L\in\Pic\_0(X)$ with non-zero $\Ext^1(F,L(-n))\cong\C$ defining an extension $I\in I_m(X,\beta)\times\Pic\_0(X)$,
\begin{equation}\label{extn}
0\To L(-n)\To I\To F\To0.
\end{equation}
As noted in the proof of Proposition~\ref{prop.converse}, $I$ and $L(-n)[1]$ are $\nu_{b_0,w_0}$-stable of the same phase, so $\Ext^1(I,L(-n))=0$. 
Therefore, applying $\Ext(\ \cdot\ ,L(-n))$ to \eqref{extn} gives
\begin{equation}\label{extgroups}
  \Ext^i(F,L(-n))\ =\
  \left\{\!\!\begin{array}{lcl}\C, && i=1, \\ 0, && i\le0\text{ or }i\ge4.\end{array}\right.
\end{equation}
We would like to do this in families, as the pairs $(F,L)$ move over $M_{X,H}(\vi_n)\times\Pic\_0(X)$. But $\Ext^1(F,L(-n))$ is the non-zero Ext group of lowest degree by \eqref{extgroups}, so basechange issues mean it does not show up in the relative $\ext$s of the family version. Instead, we use its Serre dual
$$
H^2(F\otimes L^*(n)\otimes K_X)\ \cong\ \Ext^1(F,L(-n))^*.
$$
To set up its family version \eqref{Edot} below, we let $\cF$ be a universal twisted sheaf\,\footnote{Working with twisted sheaves is no harder than working with ordinary sheaves; the formalism is set up in \cite{Ca}, for instance. Eventually we will be able to remove the twisting to make $\cF$ a coherent sheaf.}
 over $X\times M_{X,H}(\vi_n)$ and let $\cL$ be a Poincar\'e sheaf on $X\times\Pic\_0(X)$. Suppressing some obvious pull back maps for clarity and pushing forward along the map
 \begin{equation}\label{pie}
 X\times M_{X,H}(\vi_n)\times\Pic\_0(X)\xrightarrow{\ \pi\ }M_{X,H}(\vi_n)\times\Pic\_0(X),
 \end{equation}
we consider the twisted sheaf
\begin{equation}\label{Edot}
\cG\ :=\ R^2\pi_*\left(\cF\otimes\cL^*(n)\otimes K_X\right) \quad\text{on}\quad M_{X,H}(\vi_n)\times\Pic\_0(X).
\end{equation}
By Serre duality applied to \eqref{extgroups}, there are no higher-degree push down cohomology sheaves, so basechange applies to show that on restriction to any closed point $(F,L)\in\supp\cG$,
$$
\cG\big|_{(F,L)}\ =\ H^2(F\otimes L^*(n)\otimes K_X)\ \cong\ \Ext^1(F,L(-n))^*\ \cong\ \C.
$$
Therefore, $\cG$ is a (twisted) line bundle on its support $S_\cG$, where
$$
\xymatrix{S_\cG\,:=\,\supp\cG\ \ar@{^(->}[r]<-.2ex>^-\iota\ar[dr]& M_{X,H}(\vi_n)\times\Pic\_0(X) \ar[d]^p \\
& M_{X,H}(\vi_n)}
$$
is a set-theoretic section of $p$, \textit{i.e.}\ a single point in each $\Pic\_0(X)$ fibre over $M_{X,H}(\vi_n)$. We want to upgrade this statement to one about schemes instead of sets.

\begin{Lem} The support $S_\cG$ of\, $\cG$ is a section of $p$, so is scheme-theoretically isomorphic to $M_{X,H}(\vi_n)$.
\end{Lem}

\begin{proof} We first prove $S_\cG\to M_{X,H}(\vi_n)$ is an embedding by showing that the fibre over any closed point $F\in M_{X,H}(\vi_n)$ is a \emph{reduced} point $L\in\Pic\_0(X)$. Here $L$ is the unique line bundle such that $\Ext^1(F,L(-n))$ is non-zero. Let $e$ be a generator of $\Ext^1(F,L(-n))$, defining the extension \eqref{extn}. Applying $\Ext^*(\ \cdot\ ,L(-n))$ to \eqref{extn} gives an isomorphism
\begin{equation}\label{isomorphi}
\Ext^1\!\left(L(-n),L(-n)\right)\xrightarrow[\cup e]{\ \isolow\ }\Ext^2\!\left(F,L(-n)\right)
\end{equation}
because $\Ext^{\le2}(I,L(-n))=0$ for $I\in I_m(X,\beta)\times\Pic\_0(X)$ and $n\gg0$.
This map takes any first-order deformation of $L$ in $\Pic\_0(X)$ to the obstruction (in the right-hand group) to deforming $e\in\Ext^1_X(F,L(-n))$ with it. Thus $e$ is totally obstructed --- it does not deform to first order with $L$. That is, the Zariski tangent space to the fibre of $S_\cG$ over $\{F\}$ --- the kernel of \eqref{isomorphi} --- is trivial. This shows that $S_\cG\to M_{X,H}(\vi_n)$ is an embedding. \medskip

To prove $S_\cG\to M_{X,H}(\vi_n)$ is an isomorphism of schemes, it is now sufficient to show its basechange to any fat point of $M_{X,H}(\vi_n)$ is an epimorphism. That is, take the maximal ideal $\m\subset\cO_{M_{X,H}(\vi_n)}$ at $F$, set
$$
M_k\ :=\ \Spec\left(\cO/\m^k\right),
$$
and assume inductively that we have proved $S_k:=S_\cG\times\_{M_{X,H}(\vi_n)}M_k\to M_k$ is an epimorphism (and so an isomorphism). We want to show the same is true for $k+1$.

Let $F_k,\,L_k$ be the restrictions of the universal sheaves to $X\times M_k\times\Pic\_0(X)$. By our inductive assumption and basechange, we know that $\cG|_{S_k}$ is a line bundle on $S_k$. By relative Serre duality down $\pi_k$ --- the basechange of $\pi$ \eqref{pie} to $S_k$ --- its dual is the line bundle
$$
\ext^1_{\pi_k}\!\left(F_k,L_k(-n)\right)\quad\text{on}\quad S_k\,\cong\,M_k, 
$$
which therefore has a trivialising section $e_k\in\Ext^1(F_k,L_k(-n))$ defining an extension
\begin{equation}\label{Ik}
0\To L_k(-n)\To I_k\To F_k\To 0.
\end{equation}
Since $\Pic\_0(X)$ is smooth, the classifying map $M_k\to\Pic\_0(X)$ of $L_k$ can be extended to $M_{k+1}\to\Pic\_0(X)$, thus defining a preliminary $L_{k+1}$ over $X\times M_{k+1}$. We already have $F_{k+1}:=\cF|_{X\times M_{k+1}}$. The obstruction to extending the extension class $e_k$ to any
$$
e_{k+1}\ \in\ \Ext^1_{X\times M_{k+1}}\!\left(F_{k+1},L_{k+1}(-n))\right)
$$
is a class ob in
\begin{equation}\label{key}
\Ext^2_{X\times M_k}\!\left(F_k,L_k(-n)\otimes\frac{\m^k}{\m^{k+1}}\right)\xleftarrow[\ \cup e_k]{\isolow}\Ext^1_{X\times M_k}\!\left(L_k(-n),L_k(-n)\otimes\frac{\m^k}{\m^{k+1}}\right).
\end{equation}
Here the isomorphism follows from applying $\Ext^*\!\big(\,\ \cdot\,\ ,L_k(-n)\otimes(\m^k/\m^{k+1})\big)$ to \eqref{Ik}; \textit{cf.}~\eqref{isomorphi}. Now the space of choices of $L_{k+1}$ extending $L_k$ (\textit{i.e.}\ maps $M_{k+1}\to\Pic\_0(X)$ extending the given map from $M_k$) is a torsor over the right-hand group of \eqref{key}. Therefore, the class of $(-\,$ob) in this group defines a new $L_{k+1}$ for which the obstruction to the existence of $e_{k+1}$ now vanishes. This $e_{k+1}$ then trivialises
$$
\ext^1_{\pi_{k+1}}\left(F_{k+1},L_{k+1}(-n)\right)\quad\text{on}\quad M_{k+1},
$$
showing it is a line bundle and that we have defined an extension $S_{k+1}\subset S_\cG$ of $S_k\subset S_\cG$. Thus $S_\cG\to M_{X,H}(\vi_n)$ is an epimorphism after basechange to $M_{k+1}$, as required.
\end{proof}

By basechange $\cG=\iota_*\cT$, where $\cT$ is the twisted line bundle
$$
R^2\left(\pi|\_{X\times S_\cG}\right)_*\left((\cF\otimes\cL^*(n)\otimes K_X)\big|_{X\times S_\cG}\right) \quad\text{on}\quad S_\cG\,\cong\,M_{X,H}(\vi_n).
$$
Denote the composition of $\iota$ with the projection to $\Pic\_0(X)$ by $f\colon M_{X,H}(\vi_n)\to\Pic\_0(X)$, and let \linebreak $\pi'\colon X\times M_{X,H}(\vi_n)\to M_{X,H}(\vi_n)$ be the projection. Identifying $S_\cG$ with $M_{X,H}(\vi_n)$, the above becomes
$$
\cT\ =\ R^2\pi'_*\left(\cF\otimes f^*\cL^*(n)\otimes K_X\right) \quad\text{on}\quad M_{X,H}(\vi_n).
$$
So replacing $\cF$ by $\cF\otimes(\pi')^*\cT^*$ gives a new universal \emph{sheaf} (the twistings cancel) such that, by relative Serre duality down $\pi'$,
$$
\ext^1_{\pi'}\left(\cF,f^*\cL(-n)\right)\ \cong\ \cO_{M_{X,H}(\vi_n)}.
$$
The section $1\in\Gamma(\cO)$ defines an extension 
$$
0\To f^*\cL(-n)\To\cI\To\cF\To0.
$$
Since $\cI\otimes f^*\cL^*$ is flat over $M_{X,H}(\vi_n)$, we get a family of Joyce--Song pairs classified by a map $M_{X,H}(\vi_n)\to\js_n(\vi)$. By construction, its product with $f$ is the inverse of \eqref{morph}.
\end{proof}

\section{Relationship to the work of Toda}\label{related}

This paper, its predecessor \cite{FT} and its sequel \cite{F22} use methods pioneered by Yukinobu Toda. In \cite{TodaBG} he also studied 2-dimensional sheaves on threefolds $X$ satisfying the Bogomolov--Gieseker inequality, under the additional assumption that $X$ is Calabi--Yau with Pic$\,X=\Z$. Like us, he starts in the large volume region and then moves down a vertical line in the space of weak stability conditions to find walls of instability by applying the Bogomolov--Gieseker inequality to weakly semistable objects. In this way, he gave a mathematical formulation and proof of Denef--Moore's version \cite{DM} of the famous OSV conjecture \cite{OSV}.

In our work we move down the same vertical line $\{b=b_0\}$ but diverge from Toda’s method in two main ways: 
\begin{itemize}
    \item Toda uses $\Pic(X)=\Z$ and the Bogomolov--Gieseker inequality at the point $(b_0,w_0)$ of $\ell\cap U$ to constrain the Chern characters of the destabilising objects $F_1,\,F_2$. Instead, we employed a wall-crossing argument to analyse $F_1,\,F_2$ along $\ell\cap U$, using the fact that they stay in $\cA(b)$ to constrain $\ch(F_i)$ (Proposition~\ref{prop.the first wall}). Further, we then moved down $\{b=b'\}$, showing $F_1$ remains semistable to apply the Bogomolov--Gieseker inequality to it at $(b^*,w^*)$ as in \eqref{b*w*}. This gave a stronger bound for $\ch_3(F_1)$. A similar argument (replacing $b'=-\frac1{H^3}$ by $-n+\frac1{H^3}$ as in \cite[Section 8]{FT}) did the same for $\ch_3(F_2(n))$ (Propositions~\ref{prop. upper bound for ch_3(F_1)} and~\ref{prop. upper bound for ch_3(F_2)}). Together, these completely specified $\ch_2(F_1)\cdot H = \beta\cdot H$ (Proposition~\ref{prop. exact value of ch2}).
    \item In turn this allows us to show that \emph{all} semistable sheaves of class $v_n$ are destabilised by Joyce--Song pairs on \emph{the first wall}. Since Toda does not take $n\gg0$ as large as we do, he also has to analyse many subsequent walls.
\end{itemize}
As a result, our wall-crossing formula \eqref{Theorem 2 equality} of Theorem~\ref{Theorem 2} is much simpler than Toda’s. If we specialise his result to our situation by fixing his parameters $\xi=2,\,\mu = 12\frac{\beta\cdot H}{H^3} + \frac{2}{H^3}$ and taking $n \gg 0$ (while noting that his Conjecture 1.4 has now been proved, see~\cite{BBBJ}), his wall-crossing formula becomes the following.

\begin{Thm}[\textit{cf.}~\protect{\cite[Theorem 3.18]{TodaBG}}]\label{Theorem.toda.1}
Let $X$ be a smooth projective Calabi--Yau threefold such that $\Pic(X) = \mathbb{Z}\cdot H$ and Conjecture~\ref{conjecture} holds. Fix $n\gg0$, and let
\begin{equation}\label{coneC}
C\ :=\ \left\{(\beta_i, m_i)\,\in\,H_2(X) \oplus H_0(X)\ \colon\, \beta_i\cdot H \,\leq\,6 \beta\cdot H, \ \,\abs{m_i}\,<\,(6 \beta\cdot H +1)n  \right\}.
\end{equation}
Then $\Omega_{\vi_n}(X)$ is given by
\begin{equation}\label{toda's seri.2}
\sum_{\substack{(\beta_i, m_i) \,\in\, C,\ \beta_2-\beta_1 = \beta,\\m_1-m_2 -n \beta_1\cdot H = m}}
\hspace{-5mm} (-1)^{\chi(\vi(n)) - n \beta_1\cdot H -1}\left(\chi(\vi(n)) - n \beta_1\cdot H\right)I_{m_2, \beta_2}(X)\,P_{-m_1, \beta_1}(X). 
	\end{equation}
\end{Thm}

Here $P_{m,\beta}(X)$ is the stable pairs invariant, see \cite{PT}, the degree of the virtual cycle on the moduli space $P_m(X,\beta)$ of stable pairs  $(F,s)$ with $\chi(F)=m$ and $[F]=\beta$.

Toda pointed out to us how \eqref{toda's seri.2}
 can be made compatible with our simplification \eqref{Theorem 2 equality}. By another application of the Bogomolov--Gieseker-type inequality one can prove Castelnuovo-type bounds to show that $P_k(X,\beta_1)$ and $I_k(X,\beta_2)$ are empty for $k$ sufficiently small. Since the bounds $\beta_i\cdot H\le 6\beta\cdot H$ in the definition \eqref{coneC} of the cone $C$  are independent of $n$, we can therefore choose a uniform $n\gg0$ so that each term in the sum \eqref{toda's seri.2} has at least one of $P_{-m_1}(X,\beta_1)$ or $I_{m_2}(X,\beta_2)$ empty for $m_1-m_2=n\beta_1\cdot H+m$ (unless $\beta_1=0=m_1$). This would give another (ultimately lengthier) proof of Theorem~\ref{Theorem 2} when $X$ is a Calabi--Yau threefold with $\Pic=\Z$.
 
 In \cite{F22} the first author extends our methods and Toda’s to prove an OSV-like result for general Calabi--Yau threefolds, without the $\Pic(X) = \Z$ condition.

\section{Modularity}\label{modular}

On Calabi--Yau threefolds, the invariants $\Omega_{v_n}(X)$ are expected to have modular properties. There are two points of view on this:  one physical (``S-duality'') and one mathematical (Noether--Lefschetz theory). We describe these now on a Calabi--Yau threefold $X$ with $H^1(\cO_X)=0$ and $H^2(X,\Z)_{\tors}=0$ for simplicity.

\subsection{S-duality}\label{S-duality}
Physicists have long conjectured that counts of D4-D2-D0 branes should have modular properties; see \cite{MSW, GSY, al, DM}. In \cite{GSY} the proposal was to use Gieseker stable sheaves, \textit{i.e.}\ the invariants $\Omega(v_n):=\Omega_{v_n}(X)$. Some suggestive examples on the quintic threefold were calculated and shown to be compatible with the conjecture in \cite{GY}. Over time the conjecture has evolved somewhat; see \cite{AMP} for the state of the art (and extension to refined counting invariants). It is now expected that one should replace Gieseker stability by stability at the ``\emph{large volume attractor point}'' for the charge $v_n$. Here the central charge of $E$ can be found by pairing with minus the exponential of minus the complexified K\"ahler form of \cite[Equation 2.6]{AMP}, giving 
$$
\frac12\lambda^2n^2H^2.\ch_1(E)+i\lambda\left(\!\ch_2(E)\cdot nH-\ch_1(E).\left(\beta+\frac12n^2H^2\right)\right) +o(\lambda)
$$
to leading order in their parameter $\lambda\to\infty$. After scaling and adding a constant, this corresponds to the slope function
$$
\frac{\ch_2(E)\cdot H}{\ch_1(E)\cdot H^2}\,-\,\frac1n\cdot\frac{\ch_1(E).\beta}{\ch_1(E)\cdot H^2}\,.
$$
As $n\to\infty$ with $E$ fixed, this tends to $\nu\_H(E)$, defined in \eqref{nuslope}, and by Theorem~\ref{theorem.1} $M_{X,H}(v_n)$ is precisely the moduli space of $\nu\_H$-stable sheaves. Furthermore, there are no strictly $\nu\_H$-semistable sheaves, so we can perturb $\nu\_H$ a little without changing this result. However, the sheaves whose stability we test also depend on $n$, so this argument is suggestive but not a proof that sheaves in $M_{X,H}(v_n)$ might be ``attractor stable'' (and describe \emph{all} attractor semistable sheaves of class $v_n$) for large $n$. So we might expect the invariants $\Omega(v_n)$ to be the ``MSW invariants'' of \cite{MSW, AMP}. (We return to this point in Remark~\ref{rmk}.)

Although Gieseker stability is not always preserved by tensoring by a line bundle, slope stability is. Therefore, by Theorem~\ref{theorem.1}, for $n\gg0$ we have
\begin{equation}\label{invariance}
\Omega(v_n)\ =\ \Omega\left(e^\ell v_n\right)\,\text{ for all }\,\ell\in H^2(X,\Z),
\end{equation}
where $e^\ell v_n$ is the cup product of $e^\ell\in H^*(X,\Q)$ with $v_n$. Note that $e^\ell v_n$ has the same $H^2$ class $nH$ as $v_n$, but $H^4$ class 
$$
-\beta-\frac12n^2H^2+nH\cdot\ell.
$$
Therefore, the invariance \eqref{invariance} shows the data of all invariants $\Omega(v_n)$, over all $\beta$ and $m$ (for fixed $n\gg0$),\footnote{Since we choose $n\gg0$ only after fixing $m$, we may need to truncate our generating series, considering bounded $m\le M(n)$ for a given $n$. We return to this issue in Remark~\ref{rmk}.} is in fact captured in the smaller set of invariants $\Omega(0,nH,\ch_2,\ch_3)$ for
\begin{equation}\label{H4H2}
\ch_2+\frac12n^2H^2\ \in\ \frac{H^4(X,\Z)}{nH\cup H^2(X,\Z)}\ =:\ \Gamma\,.
\end{equation}
The group $\Gamma$ is finite by the hard Lefschetz isomorphism $\cup\,nH\colon H^2(X,\Q)\xrightarrow\isolow H^4(X,\Q)$. We let $\beta/nH$ denote the inverse image of $\beta$ under this map. Therefore, all the enumerative information can be encoded in the vector of generating series
\begin{equation}\label{modu}
\bigoplus_{\beta\in\Gamma}h_{nH,\beta}(q), \quad h_{nH,\beta}(q)\ :=\ \sum_{\widehat m}\Omega\left(0,nH,-\beta-\frac12n^2H^2,-m+\frac16n^3H^3\right)q^{\widehat m},
\end{equation}
where $\widehat m$ is the following normalisation of $\ch_3$: 
\begin{equation}\label{mhat}
\widehat m\ :=\ m+\frac12nH\cdot\beta-\frac1{24}nH\cdot c_2(X)-\frac1{24}n^3H^3+\frac12\int_X\frac{\beta}{nH}\cup\beta,
\end{equation}
which is easily checked to be invariant under $v_n\mapsto e^\ell v_n$. The series \eqref{modu} are the product of Laurent series in $q$ with a prefactor $q^c,\,c\in\Q$. Setting $q=e^{2\pi i\tau}$, we think of them as meromorphic functions of $\tau$ in the upper half plane. In \cite{GSY}, \eqref{modu} was conjectured to be a vector-valued modular form of weight $-b_2(X)-\frac12$.
This is now expected to be true only for irreducible $\ch_1$, which is far from our case of $\ch_1=nH$.

For more general $\ch_1$, the current expectation is that \eqref{modu} should be a vector-valued \emph{mock modular form} of depth $k-1$, where $k$ is the maximum over all non-trivial decompositions
\begin{equation}\label{DDk}
D\=D_1+\dots+D_k\quad\text{ for all divisors }\,D\in|\cO(n)|.
\end{equation}
That is, it should admit a non-holomorphic modular completion $\bigoplus_{\beta\in\Gamma}\widehat h_{nH,\beta}(q)$ made from $k-1$ iterated Eichler integrals involving the functions 
$h_{[D_1],\beta_1}(q),\dots,h_{[D_k],\beta_k}(q)$. Explicit formulae for the $\widehat h$ in terms of $h$ are given in \cite[Equation 2.11]{AMP} and inverted to express $h$ in terms of $\widehat h$ in \cite[Equation 2.15]{AMP}. Under the modular group, the $\widehat h$ should transform as in \cite[Equation 2.10]{AMP} with weight $-\frac12b_2(X)-1$. That is, $\widehat h_{nH,\beta}(-1/\tau)$ should be
$$
-\frac{(-i\tau)^{-\frac12b_2(X)-1}}{\sqrt{|\Gamma|}}\exp\left(-2\pi i\left(\frac14 n^3H^3-\frac18c_2(X)\cdot nH\right)\right)\sum_{\gamma\in\Gamma}\exp\left(-2\pi i\int_\gamma\frac{\beta}{nH}\right)\,\widehat h_{nH,\gamma},
$$
and
$$
\widehat h_{nH,\beta}(\tau+1)\ =\ \exp\left(2\pi i\left(\frac1{24}c_2(X)\cdot nH+\frac12\int_{\beta}\,\frac{\beta}{nH}+\dfrac12\beta\cdot nH+\dfrac18n^3H^3\right)\right)\,\widehat h_{nH,\beta}(\tau).
$$

It is further predicted that, apart from their poles of order $\frac1{24}((nH)^3+c_2(X)\cdot nH)$ at $q=0$, the functions $h$ and $\widehat h$ are bounded. Since they are vectors of length $\dim\Gamma=nH^3$ and have modular weight $-b_2(X)/2-1$, the dimension of the relevant space of (mock) modular forms can be analysed; see \cite{Man}. In our case its dimension works out as $O(n^4)$, so the first $O(n^4)$ Fourier coefficients $\Omega(v_n)$ should determine the rest.

The hope would be to use this as a method for determining the MNOP invariants $I_{m,\beta}(X)$ for $m>O(n^4)$ in terms of those with smaller $m$. This does not currently work in general because once $m$ becomes large, we have to increase $n$ in Theorem~\ref{Theorem 2} to get the relationship between $\Omega(v_n)$ and $I_{m,\beta}(X)$. We would need the bound $n\gg0$ required in Theorem~\ref{Theorem 2} to be improved to $n>O(m^{1/4})$ or better. But very recent work \cite{AFKPS} manages to improve the bounds for small $m,n$ in specific examples and thus use modularity to calculate new Gromov--Witten invariants.

\subsection{Noether--Lefschetz theory}\label{N-L theory}
Here we flesh out a suggestion of Davesh Maulik to explain, or perhaps one day prove, the modularity properties of the generating series of invariants $\Omega(v_n)$ directly. We thank Luis Garcia for his insight and generous expert assistance with this section.

Again let $X$ be a Calabi--Yau threefold with $H^1(\cO_X)=0$, and again we work with bounded $m$ and then large $n\gg0$. We return to this point in Remark~\ref{rmk}.  By Theorem~\ref{theorem.1} all sheaves in $M_{X,H}(v_n)$ are of rank 1 on their scheme-theoretic support, and that support is a divisor $D\in|\cO(n)|$. The generic $\iota\colon D\into X$ is smooth and supports precisely the stable sheaves
\begin{equation}\label{sheef}
\iota_*(L\otimes\cI_Z),
\end{equation}
where $L$ is a line bundle on $D$ and $Z\subset D$ is a 0-dimensional subscheme.
The existence of $L$ means $D$ lies in one of the Noether--Lefschetz loci $NL_{d,\beta}\subset|\cO(n)|$ of divisors containing an integral $(1,1)$ class $\ell:=c_1(L)$ such that $\iota_*\hspace{.6pt}\ell=\beta$ and the discriminant of the sublattice $\,\langle\ell,h\rangle\subset H^2(D,\Z)$ is $d$. Here $h:=H|_D$ and
\begin{equation}\label{discri}
d\=\disc\langle\ell,h\rangle\ :=\ h^2\ell^2-(h\cdot \ell)^2,
\end{equation}
where the intersections are taken on $D$. (Of course, $h^2$ and $h\cdot\ell$ can be expressed on $X$ as $nH^3$ and $H\cdot\beta$, respectively, but $\ell^2$ cannot be determined by its image $\beta=\iota_*\hspace{.6pt}\ell$ in $X$).	

We briefly review some Noether--Lefschetz theory. We suppose $H^2(X,\Z)=\Z\cdot H$ for simplicity. Set $\Lambda:=\langle h\rangle^\perp\subset H^2(D,\Z)$ to be the primitive cohomology. The Lefschetz theorems give $H^4(X,\Z)\cong H^2(D,\Z)/\Lambda$, which surjects onto $\Lambda^*/\Lambda$ by the unimodular intersection pairing on $D$. The kernel is $\langle h\rangle$, so we can describe the finite group $\Gamma$ from \eqref{H4H2} as
\begin{equation}\label{L*L}
\Gamma\ =\ \frac{H^4(X,\Z)}{nH\cup H^2(X,\Z)}\ \cong\ \frac{\Lambda^*\!}{\Lambda}\ \cong\ \Z/N\Z\,,
\end{equation}
where $N=nH^3$. Then, up to the action of $\Aut\big(H^2(D,\Z),h\big)$, the data of the 2-dimensional sublattice $\langle \ell,h\rangle\subset H^2(D,\Z)$ is equivalent to the data of its discriminant disc and its \emph{coset} --- the image of $\ell$ in the quotient of the group $\Lambda^*/\Lambda$ from \eqref{L*L} by $\pm1$.

We have a map $\Phi_n\colon D\mapsto X_
\Lambda$ from the open set $|\cO(n)|^\circ$ of smooth divisors $D$ to the moduli space\footnote{Here $X_\Lambda$ is the quotient of the period domain by $\Aut(H^2(D,\Z),h)=\ker(\Aut\Lambda\to\Aut(\Lambda^*/\Lambda)).$} of weight 2 polarised Hodge structure of signature $(h^2(\cO_D),h^{1,1}(D)-1)$.  This moduli space $X_
\Lambda$ contains universal Noether--Lefschetz loci\footnote{These are called Hodge loci in the paper \cite{Ga}, which extends results of Borcherds and Kudla--Millson from hermitian symmetric spaces to the period domains of interest to us.} $\mathsf{NL}_{d,\gamma}$ consisting of Hodge structures on $\Lambda$ admitting a $(1,1)$ vector in $\Lambda^*$ of square $d/h^2$ and coset $\gamma$. (The link to 2-dimensional sublattices $\langle\ell,h\rangle\subset H^2(D,\Z)$ takes $\ell\in H^2(D,\Z)$ to its projection $\ell-\frac{\ell\cdot h}{h^2}h\in\Lambda^*$ orthogonal to $h$. This has square $\ell^2-\frac{(h\cdot\ell)^2}{h^2}=\frac d{h^2}$, where $d$ is the discriminant \eqref{discri}.)

Since the dimension of $|\cO(n)|^\circ$ matches the codimension of the Hodge loci $\mathsf{NL}_{d,\gamma}$,
$$
\dim|\cO(n)|^\circ\=h^0(\cO_X(n))-1\=h^2(\cO_D)\=\codim\mathsf{NL}_{d,\gamma},
$$
we could imagine defining their intersection by pulling back the Thom forms of the Hodge loci constructed in \cite{Ga} and integrating over $|\cO(n)|^\circ$. Below we will come back to the obvious problems of non-compactness of the Hodge loci in showing such integrals converge; for now we ignore them and just work with smooth $D$ in the interior of the period domain.

The constraints of Griffiths transversality mean we can probably never expect the intersection of $\Phi_n|\cO(n)|^\circ$ and $\mathsf{NL}_{d,\gamma}$
to be of the correct dimension 0. However, in order to formulate a conjecture, one can imagine perturbing the complex structure on $X$ to a non-integrable almost complex structure (compatible with the symplectic structure dual to $H$) to ensure the intersection is 0-dimensional. Since the virtual-dimension-0 deformation theory of $\iota_*L$, or of the pair $(D,L)$, can be matched with the deformation theory of the intersection (see \cite[Section 2.1]{KT1}, for instance), we would, as usual, expect to be able to avoid such non-algebraic deformations by working \textit{in situ} with the virtual cycle, yielding the same intersection numbers via excess intersection.


So we imagine $\Phi_n^*[\mathsf{NL}_{d,\gamma}]$ reduced isolated intersection points $(D,L)$ with an extra Hodge class $\ell=c_1(L)$ of discriminant disc$\,\langle\ell,h\rangle=d$ and coset $\gamma$. Each such point would generate a component of a moduli space $M_{X,H}(v_n)$ given by
\begin{equation}\label{hilbk}
\Hilb^kD\ \cong\ \left\{L\otimes\cI_Z\ \colon\,|Z|=k\right\}
\end{equation}
parameterising the sheaves \eqref{sheef}. Here we take the charge $m$ in $v_n$, see \eqref{vn}, to be
$$
m\ =\ |Z|+\frac12\left(\beta\cdot nH-\ell^2\right)\ =\ k+\frac12\beta\cdot nH-\frac{(H\cdot\beta)^2}{2nH^3}-\frac d{2h^2},
$$
by calculating $\ch_3=\frac16n^3H^3-m$ of \eqref{sheef}. Taking the Euler characteristic of \eqref{hilbk}, weighting by $q^{\widehat m}$, with $\widehat m$ as in \eqref{mhat}, and summing gives, by the G\"ottsche formula, the generating series
\begin{equation}\label{genser}
q^{c-\frac d{2(h)^2}}\left(q^{-\frac1{24}}\prod_{\,i=1}^\infty\frac1{(1-q^k)}\right)^{\!e(D)}=\ q^{c-\frac d{2h^2}\,}\eta(q)^{-e(D)},
\end{equation}
where $e(D)$ is the topological Euler characteristic of any smooth member $D\in|\cO(n)|$,
$$
e(D)\ =\ c_2(X)\cdot nH+n^3H^3,
$$
and $c=nH\cdot\beta-\frac{(H\cdot\beta)^2}{2nH^3}+\frac12\int_\beta\frac\beta{nH}\in\Q$. (As in the last section, we should really truncate this sum over $m\le M(n)$ if we want to use Theorem~\ref{theorem.1} to identify $M_{X,H}(v_n)$ with unions of Hilbert schemes \eqref{hilbk}, but see Remark~\ref{rmk}.)

Summing \eqref{genser} over the divisors $D$, by summing over all discriminants and cosets in \eqref{L*L}, the vector of generating series \eqref{modu} becomes
\begin{equation}\label{final2}
q^c\eta(q)^{-e(D)}\bigoplus_{\gamma\in\Gamma}\,\sum_d\Phi_n^*\left[\mathsf{NL}_{d,\gamma}\right]\,q^{-\frac d{2h^2}}\,.
\end{equation}
Now $\eta(q)^{-e(D)}$ is modular of weight $-\frac12e(D)$, and  $\bigoplus_{\gamma\in\Gamma}\,\sum_d\big[\mathsf{NL}_{d,\gamma}\big]\,q^{-\frac d{2h^2}}$ is a vector-valued modular form (with values in the cohomology of the moduli space $X_{\Lambda}$ of Hodge structures) of weight $\frac12\dim H^2_{\prim}(D)=\frac12(e(D)-3)$, by \cite[Theorems 1.2 and 5.2]{Ga}.\footnote{Note that our $\frac d{2h^2}=\frac12\big(\ell-\frac{\ell\cdot h}{h^2}h\big)^2$, see \eqref{discri}, corresponds to Garcia's $\frac12\langle v,v\rangle$, see \cite[Theorem 1.2]{Ga}, on setting $v=\ell-\frac{\ell\cdot h}{h^2}h\in\Lambda^*$ to be the projection of $\ell$ to $\langle h\rangle^\perp\otimes\Q$.}

So we conclude that \eqref{final2} is modular of total weight $-\frac{3}{2}$ if we can make finite sense of $\Phi_n^*\big[\mathsf{NL}_{d,\gamma}\big]$. This will involve further work studying degenerations of Hodge structure at the boundary of the space of Hodge structures. It is natural to expect non-holomorphic corrections, turning modular forms into mock modular forms. One point of view is that the Thom forms of the Hodge loci are not precisely holomorphic --- taking $\dbar$ gives exact forms $da$ on the moduli space of Hodge structures, see \cite[Equation 4.39]{Ga}, which are therefore exact on pull back to $|\cO(n)|^\circ$ but may not be on the boundary of $|\cO(n)|$  (where $d(\Phi_n^*a)$ may have poles with non-zero residues). In a special case ($\Sym^2$ of the Hodge structures of elliptic curves), the boundary and convergence analysis was carried out successfully in \cite{Fu} and indeed found to give mock modular forms.

In particular, taking account of reducible and non-reduced $D$ at the boundary of $|\cO(n)|$ will necessarily add cross-terms involving all non-trivial decompositions
$$
D\=D_1+\dots+D_k\quad\text{ for all divisors }\,D\in|\cO(n)|.
$$
This is the same data  used in \eqref{DDk} to generate the non-holomorphic mock modular completions $\widehat h$ of the generating series $h$ from \eqref{modu}.
So it seems reasonable to expect the Noether--Lefschetz story to be compatible with, or one day even prove, S-duality. Gholampour and Sheshmani have been exploring example calculations along related lines in recent years; see for instance  \cite{GST, GS}.

\begin{Rem}\label{rmk}
In our modularity discussions of the last two sections, two issues have arisen which we consider to be related. In Section~\ref{S-duality} it was not clear we had the right stability condition for our invariants $\Omega(v_n)$ to be the MSW invariants. In both Sections~\ref{S-duality} and~\ref{N-L theory} there was the issue that our description of $M_{X,H}(v_n)$ in Theorem~\ref{theorem.1} was only valid for bounded $m\le M(n)$.

Given the discussion in this section, it seems natural to suggest the solution to both problems should be the following. We should take moduli spaces of sheaves $\iota_*(L\otimes I_Z)$ \eqref{sheef} \emph{for any $m$ and $n$}, and their invariants should have mock modular generating series. In other words, we expect that the physicists' attractor stable objects should be precisely the sheaves \eqref{sheef} with rank 1 on their support, independently of $m,n$. (When $D$ is non-reduced or irreducible, we should also use a stability condition on the line bundle $L$, probably $\nu\_H$-slope stability.) Their virtual counts would then be the MSW invariants. 

For small $m\le M(n)$ Theorem~\ref{theorem.1} gives precisely the sheaves $\iota_*(L\otimes I_Z)$ \eqref{sheef} with $L^*$ effective.\footnote{The notation is different; here $L^*$ is the line bundle corresponding to the divisorial part of $C$ in \eqref{form} and Theorem~\ref{theorem.1}. Since we are assuming $H^2(X,\Z)=\Z$, the line bundles $L,T$ in Theorem~\ref{theorem.1} are trivial.} For $m\in(M_1(n),M_2(n)]$, we find $M_{X,H}(v_n)$ parameterises sheaves of the same form $\iota_*(L\otimes I_Z)$, see \eqref{sheef},  but with $L=\cO_D(C_1-C_2)$ possibly non-effective; this is proved in \cite{TodaBG} when $\Pic X=\Z$ and \cite{F22} in general. For $m>M_2(n)$ we expect to have to change the stability condition to get a moduli space consisting of only the sheaves \eqref{sheef} and to get the MSW invariants.
\end{Rem}

\newcommand{\etalchar}[1]{$^{#1}$}

\end{document}